\documentclass[3p,times,11pt,leqno]{elsarticle}

\usepackage{microtype}
\usepackage[english]{babel}

\usepackage{latexsym}
\usepackage{amssymb}
\usepackage{amsmath}
\usepackage{fleqn}
\usepackage[all]{xy}
\SelectTips{eu}{}
\usepackage{booktabs}
\usepackage{xcolor}

\usepackage{pdfsync}
\usepackage[active]{srcltx}

\usepackage[colorlinks]{hyperref}

\usepackage{amsthm}
\newtheorem{definition}{Definition}[section]
\newtheorem{example}{Example}[section]

\newtheorem{proposition}{Proposition}[section]
\newtheorem{fact}{Fact}[section]
\newtheorem{lemma}{Lemma}[section]
\newtheorem{remark}{Remark}[section]
\newtheorem{corollary}{Corollary}[section]
\newtheorem{notation}{Notation}[section]
\newcommand{\smallitems}{\setlength{\parsep}{-0.5ex plus 0.5ex}%
                        \setlength{\itemsep}{-0.5ex plus 0.5ex}%
                        }

\def\Nm{\mathit{Nm}}
\def\Fr{\mathit{Fr}}
\def\sem#1{[\![ #1 ] \!]}
\def\F{\mathcal{F}}
\def\C{\mathcal{C}}

\newcommand{\co}{\,\colon\;}
\newcommand{\ra}{\rightarrow}
\newcommand{\Ra}{\Rightarrow}

\newcommand{\op}{\mathrm{op}}

\newcommand{\impl}{\Rightarrow}

\newcommand{\Nom}{\mathrm{Nom}}
\newcommand{\NNom}{\mathit{Nom}}

\def\sen#1{#1\mbox{-}\mathrm{sen}}

\newcommand{\Sign}{\mathit{Sign}}
\newcommand{\Sen}{\mathit{Sen}}
\newcommand{\Mod}{\mathit{Mod}}
\def\A{\mathcal{A}}
\def\B{\mathcal{B}}
\def\Set{\mathbf{Set}}
\def\CAT{\mathbf{C\!A\!T}}
\def\CCAT{\mathbf{CC\!A\!T}}

\def\I{\mathcal{I}}
\def\FOL{\mathcal{F\!O\!L}}
\def\AFOL{\mathcal{A\!F\!O\!L}}
\def\PL{\mathcal{P\!L}}
\def\APL{\mathcal{A\!P\!L}}
\def\OFOL{\mathcal{O\!F\!O\!L}}
\def\MOFOL{\mathcal{MO\!F\!O\!L}}
\def\HOFOL{\mathcal{H\!O\!F\!O\!L}}
\def\HMOFOL{\mathcal{H\!MO\!F\!O\!L}}
\def\MPL{\mathcal{MP\!L}}
\def\MFOL{\mathcal{MF\!O\!L}}
\def\HPL{\mathcal{H\!P\!L}}
\def\HHPL{\mathcal{H\!H\!P\!L}}
\def\HHHPL{\mathcal{H\!H\!H\!P\!L}}
\def\HFOL{\mathcal{H\!F\!O\!L}}
\def\HHFOL{\mathcal{H\!H\!F\!O\!L}}
\def\MMPL{\mathcal{MMP\!L}}
\def\MMFOL{\mathcal{MMF\!O\!L}}
\def\MHPL{\mathcal{MH\!P\!L}}
\def\MHFOL{\mathcal{MH\!F\!O\!L}}
\def\REL{\mathcal{R\!E\!L}}

\def\BREL{\mathcal{B\!R\!E\!L}}
\def\BRELC{\mathcal{B\!R\!E\!L\!C}}
\def\SETC{\mathcal{S\!E\!T\!C}}
\def\SET{\mathit{SET}}

\DeclareMathAlphabet{\mathbb}{U}{msb}{m}{n}
\DeclareSymbolFont{ams}{U}{msa}{m}{n}
\DeclareSymbolFontAlphabet{\mathams}{ams}
\DeclareMathSymbol{\filter}{\mathams}{ams}{22}

\begin{document}

\title{Implicit Kripke Semantics and Ultraproducts in Stratified
  Institutions} 

\author{R\u{a}zvan Diaconescu}
\ead{Razvan.Diaconescu@imar.ro}
\address{Simion Stoilow Institute of Mathematics of the Romanian
Academy, Bucharest, Romania}

\journal{JLC} 

\begin{abstract}
We propose \emph{stratified institutions} (a decade old generalised
version of the theory of institutions of Goguen and Burstall) as a
fully abstract model theoretic approach to modal logic.   
This allows for a uniform treatment of model theoretic aspects across
the great multiplicity of contemporary modal logic systems. 
Moreover Kripke semantics (in all its manifold variations) is captured
in an implicit manner free from the sometimes bulky aspects of
explicit Kripke structures, also accommodating other forms of concrete  
semantics for modal logic systems. 
The conceptual power of stratified institutions is illustrated with
the development of a modal ultraproducts method that is independent of
the concrete details of the actual modal logical systems.  
Consequently, a wide array of compactness results in concrete modal
logics may be derived easily.  
\end{abstract}

\maketitle

\section{Introduction}

The model theory oriented formalisation by Goguen and Burstall
\cite{ins} of the notion of a logical system as an \emph{institution}
has started a line of important developments of adequately abstract
and general approaches to the foundations of software specifications
and formal system development (see \cite{sannella-tarlecki-book}) as
well as a modern version of very abstract model theory (see
\cite{iimt}). 
One of the main original motivations for introducing institution
theory was to respond to the explosion in the population of logics in
use in computing about three decades ago, a situation that continues
today perhaps at an accelerated pace. 
Among the logics with relevance in various areas of informatics there
is of course the family of modal logics, with its great multiplicity
of flavours. 
The recent works on `modalizations' of institutions
\cite{ks,HybridIns,QVHybrid,EncHybrid} (see also \cite{iimt}), in
which only the modalities (and eventually nominals and $@$) and Kripke
semantics are kept explicit, while the other ingredients (e.g. sorts,
functions, predicates, constraints, etc.) are abstracted away, has
intensified the quest for a fully abstract institution theoretic
approach that has the potential to address adequately the
specificities of modality and Kripke semantics while leaving none of
these explicit.  

Our paper proposes stratified institutions of \cite{strat} as a
general framework for a fully abstract approach to the semantics of
modal logic.  
In particular this means no explicit modalities, no explicit Kripke
structures, while still retaining the essence of Kripke semantics. 
Consequently a very general form of model theory uniformly applicable
to a wide range of concrete modal logic systems, either conventional
or more eccentric, can be developed.  
Results can be developed in a top-down manner with hypotheses
kept as general as possible and introduced on a by-need basis, the
whole development process being guided by structurally clean
causality.  
From the perspective of institution theory, our proposal yields an
institution theoretic structure fully capable of addressing modality.
The conventional definition of institution \cite{ins} may lack enough
structure to capture various specificities of modal logics, hence our
work can be regarded as a minimal but sufficient refinement of the
concept of institution towards modal logics. 

We illustrate the power of our concepts with the development of very
general modal-oriented ultraproducts method.
This provides rather automatically {\L}o\'{s}-style theorems
\cite{los55,chang-keisler} for a wide range of concrete modal systems,
as a puzzle of preservation results in the style of
\cite{upins,iimt,ks}.    
In conventional model theory the method of ultraproducts is renowned
as extremely powerful and pervading a lot of deep results (see
\cite{chang-keisler}, for example), many of these been lifted to 
the level of abstract institutions (see \cite{iimt}). 
Our developments may represent the beginning of a similar journey in
the realm of modality and Kripke semantics. 
From the many consequences of ultraproducts, here we focus only on
compactness results.
Hence we derive a series of modal compactness results for our
benchmark examples, this process having a generic nature.  

\subsection*{Summary and Contributions.}

\begin{enumerate}

\item We recall briefly some category and institution theoretic
  concepts and notations that are necessary for our paper. 

\item We from \cite{strat} the concept of stratified institution and
  slightly upgrade it. 
  Ordinary institutions arise as stratified institutions with a
  trivial stratification; in this way stratified institutions can be
  seen as more general than ordinary institutions. 
  The move in the other direction is given by two general
  interpretations of stratified institutions as ordinary
  institutions.  
  They represent high abstractions of the concepts of
  \emph{local} and \emph{global satisfaction} from modal logic,
  respectively. 

\item We provide a series of examples of stratified institutions that
  include both conventional and eccentric modal logic systems. The
  former category includes propositional and first order modal
  logic, possibly with hybrid and polyadic modalities features,
  while the latter includes the double hybridization of
  \cite{madeira-phd,EncHybrid} and a first order valuation semantics
  for first order modal logic that is based upon the `internal
  stratification' example introduced in \cite{strat}. 
  These are to be used as benchmark examples for the further
  developments in the paper. 

\item We give a straightforward extension of the well known
  institution theoretic semantics of the Boolean connectives
  $\wedge$, $\neg$, etc. and of the quantifiers $\forall$, $\exists$
  to the more refined level of stratified institutions and establish
  the relationship with their correspondents from the local and the 
  global institutions associated to the stratified institution.  

\item We introduce a semantics for modalities and for hybrid features
  in abstract stratified institutions.  
  This is one of the crucial contributions of this paper. 

\item We extend the institution theoretic method of ultraproducts
  \cite{upins,iimt} to stratified institutions. The core contributions
  here consist of a series of general preservation results across
  the abstract semantics for Boolean connectives, quantifiers,
  modalities, nominals, $@$.  
  These cover related previous developments from \cite{ks} (also to be
  found in \cite{iimt}), but with significant differences in
  generality: (1) stratified institutions with their lack of
  commitment to explicit modalities and Kripke structures are much
  more general than the `modalized' institutions of \cite{ks}; (2)
  the results of our paper cover polyadic modalities and hybrid
  features while \cite{ks} considers only the unary $\Box$ and
  $\Diamond$.  
  The above mentioned differences reflect very much in the way the
  preservation results are actually obtained. 

\item Derivation of compactness properties for the local and the
  global institutions associated to a stratified institution via
  ultraproducts. 
 
\end{enumerate}
 
\section{Category and institution theoretic preliminaries}

In this section we recall some category and institution theoretic
notions that will be used in the paper. 

We will use the diagrammatic notation for compositions of arrows in
categories, i.e. if $f \co A \ra B$ and $g\co B \ra C$ are arrows then
$f;g$ denotes their composition.  
A \emph{concrete category} $(\A,U)$ consists of a category $\A$ and a
faithful functor $U \co \A \ra \Set$.\footnote{This is most commonly
  accepted definition for concrete categories, although in 
\cite{adamek-herrlich-strecker90} this is called `concrete over
$\Set$' or `construct'.} 
A functor of concrete categories $F \co (\A,U) \ra (\B,V)$ is just a
functor $F \co \A \ra \B$ such that $U=F;V$. 
Let $\CCAT$ denote the category that has the concrete categories as
objects and functors of concrete categories as arrows. 
When it is clear from the context we may omit $U$ and simply refer to
$(\A,U)$ as $\A$. 
This implies also that for $A\in |\A|$ we may write $a\in A$ instead
of $a\in U(A)$.
We use double arrow $\Ra$ rather than single arrow $\ra$ for
natural transformations. 
A functor $\mathcal{U} \co \C \ra \C'$
\emph{preserves} a (co-)limit $\mu$ of a functor $D \co J \ra \C$ when
$\mu\mathcal{U}$ is a (co-)limit of $D;\mathcal{U}$.
It \emph{lifts} a (co-)limit $\mu'$ of  $D;\mathcal{U}$, if there
exists a (co-)limit $\mu$ of $D$ such that
$\mu\mathcal{U} = \mu'$.

The original standard reference for definitions below of institutions
and institution morphisms is \cite{ins}.  

\begin{definition}[Institution]\label{ins-dfn}
An  \emph{institution} $\I = 
\big(\Sign^{\I}, \Sen^{\I}, \Mod^{\I}, \models^{\I}\big)$ consists of 
\begin{itemize}\smallitems

\item a category $\Sign^{\I}$ whose objects are called
  \emph{signatures},

\item a sentence functor $\Sen^{\I} \co \Sign^{\I} \ra \Set$
  defining for each signature a set whose elements are called
  \emph{sentences} over that signature and defining for each signature
  morphism a \emph{sentence translation} function, 

\item a model functor $\Mod^{\I} \co (\Sign^{\I})^{\op} \ra \CAT$
  defining for each signature $\Sigma$ the category
  $\Mod^{\I}(\Sigma)$ of \emph{$\Sigma$-models} and $\Sigma$-model
  homomorphisms, and for each signature morphism $\varphi$ the
  \emph{reduct} functor $\Mod^{\I}(\varphi)$,  

\item for every signature $\Sigma$, a binary 
  \emph{$\Sigma$-satisfaction relation}
  $\models^{\I}_{\Sigma} \subseteq |\Mod^{\I} (\Sigma)|
  \times \Sen^{\I} (\Sigma)$, 

\end{itemize}
such that for each morphism 
$\varphi\co\Sigma \rightarrow \Sigma' \in \Sign^{\I}$, 
the \emph{Satisfaction Condition}
\begin{equation}
M'\models^{\I}_{\Sigma'} \Sen^{\I}(\varphi)(\rho) \text{ if and only if  }
\Mod^{\I}(\varphi)(M') \models^{\I}_\Sigma \rho
\end{equation}
holds for each $M'\in |\Mod^{\I} (\Sigma')|$ and $\rho \in \Sen^{\I} (\Sigma)$.
\[
\xymatrix{
    \Sigma \ar[d]_{\varphi} & \big|\Mod^{\I}(\Sigma)\big|
    \ar@{-}[r]^-{\models^{\I}_{\Sigma}} & 
    \Sen^{\I}(\Sigma) \ar[d]^{\Sen^{\I}(\varphi)} \\
    \Sigma' & \big| \Mod^{\I}(\Sigma')\big| \ar[u]^{\Mod^{\I}(\varphi)} 
    \ar@{-}[r]_-{\models^{\I}_{\Sigma'}} & \Sen^{\I}(\Sigma')
  }
\] 
We may omit the superscripts or subscripts from the notations of the
components of institutions when there is no risk of ambiguity. 
For example, if the considered institution and signature are clear,
we may denote $\models^{\I}_\Sigma$ just by $\models$. 
For $M = \Mod(\varphi)(M')$, we say that $M$ is the
\emph{$\varphi$-reduct} of $M'$ and that $M'$ is a
\emph{$\varphi$-expansion} of $M$.  
\end{definition} 

\begin{notation}
In any institution as above we use the following notations:
\begin{itemize}\smallitems


\item[--] for any $E \subseteq \Sen(\Sigma)$, $E^*$ denotes 
$\{ M \in |\Mod(\Sigma)| \mid M \models_\Sigma \rho$ for each $\rho \in E \}$. 

\item[--] for any $E, E' \subseteq \Sen(\Sigma)$, 
$E \models E'$ denotes $E^* \subseteq E'^*$.


\end{itemize}
\end{notation}

\begin{definition}[Compactness \cite{iimt}]
An institution $\I$ is 
\begin{itemize}\smallitems

\item[--] \emph{m-compact} when for each set $E$ of $\Sigma$-sentences,
  $E^*\not= \emptyset$ if and only if for each $E_0 \subseteq E$
  finite, $E_0^* \not= \emptyset$; 

\item[--] \emph{compact} when for each set $E$ of $\Sigma$-sentences and
  each $\Sigma$-sentence $\rho$, if $E \models_\Sigma \rho$ then there
  exists a finite $E_0 \subseteq E$ such that $E_0 \models_\Sigma
  \rho$. 

\end{itemize}
\end{definition}

\begin{definition}[Morphism of institutions]
Given two institutions $\I_i = (\Sign_i, \Sen_i, \Mod_i , \models_i)$,
with $i \in \{1, 2\}$, an institution morphism $(\Phi, \alpha, \beta)
\co \I_2 \ra \I_1$ consists of 
\begin{itemize}\smallitems

\item a signature functor $\Phi \co \Sign_2 \ra \Sign_1$,

\item a natural transformation $\alpha \co \Sen_1 \Ra \Phi;\Sen_2$,
  and 

\item a natural transformation $\beta\co \Mod_2 \Ra \Phi^{\op} ;
  \Mod_1$ 

\end{itemize}
such that the following satisfaction condition holds for 
any $\I_2$-signature $\Sigma_2$, $\Sigma_2$-model $M_2$ and
$\Phi(\Sigma_2)$-sentence $\rho$:
\[
M_2  \ \models_2 \ \alpha_{\Sigma_2}(\rho) \mbox{ \ if and only if \ }
\beta_{\Sigma_2} (M_2) \ \models_1 \ \rho.
\]

\end{definition}

The literature (e.g. \cite{iimt,sannella-tarlecki-book}) shows myriads
of logical systems from computing or from mathematical logic captured
as institutions. In fact, an informal thesis underlying institution
theory is that any `logic' may be captured by the above
definition. While this should be taken with a grain of salt, it
certainly applies to any logical system based on satisfaction between
sentences and models of any kind.
The institutions introduced in the following couple of examples will
be used intensively in the paper in various ways. 

\begin{example}[Propositional logic ($\PL$)]
\begin{rm}
This is defined as follows.
$\Sign^{\PL} = \Set$, for any set $P$, $\Sen(P)$ is generated by the
grammar
\[
S ::= P \mid S \wedge S \mid \neg S 
\]
and $\Mod^{\PL}(P) = (2^P,\subseteq)$.
For any function $\varphi \co P \ra P'$, $\Sen^{\PL}(\varphi)$
replaces the each element $p\in P$ that occur in a sentence $\rho$ by
$\varphi(p)$, and $\Mod^{\PL}(\varphi)(M') = \varphi;M$ for each
$M'\in 2^{P'}$. 
For any $P$-model $M \subseteq P$ and $\rho\in \Sen^{\PL}(P)$,
$M\models\rho$ is defined by induction on the structure of $\rho$ by
$(M \models p) = (p\in M)$, 
$(M \models \rho_1 \wedge \rho_2) = (M \models \rho_1) \wedge (M
\models \rho_2)$ and 
$(M \models \neg\rho) = \neg(M \models \rho)$.  
\end{rm}
\end{example}

\begin{example}[First order logic ($\FOL$)]
\begin{rm}
For reasons of simplicity of notation, our presentation of first order
logic considers only its single sorted, without equality, variant. 
A detailed presentation of full many sorted first order logic with
equality as institution may be found in numerous works in the
literature (e.g. \cite{iimt}, etc.).
 
The $\FOL$ signatures are pairs 
$(F=(F_n)_{n\in   \omega},P=(P_n)_{n\in \omega})$ where 
$F_n$ and $P_n$ are sets of function symbols and predicate symbols,
respectively, of arity $n$. 
Signature morphisms $\varphi \co (F,P) \ra (F',P')$ are tuples
$(\varphi^{f}=(\varphi^{f}_n)_{n\in
  \omega},\varphi^{p}=(\varphi^{p}_n)_{n\in\omega})$ such that  
$\varphi^{f}_n \co F_n \ra F'_n$ and $\varphi^{p}_n \co P_n \ra
P'_n$. 
Thus $\Sign^{\FOL}=\Set^\omega \times \Set^\omega$. 

For any $\FOL$-signature $(F,P)$, the set $S$ of the $(F,P)$-sentences
is generated by the grammar: 
\begin{equation}\label{fol-grammar}
S ::= \pi(t_1,\dots,t_n) \mid S \wedge S \mid \neg S \mid (\exists x)S'
\end{equation}
where $\pi(t_1,\dots,t_n)$ are the atoms with $\pi\in P_n$ and
$t_1,\dots,t_n$ being terms formed with function symbols from $F$, and
where $S'$ denotes the set of $(F+x,P)$-sentences with $F+x$ denoting
the family of function symbols obtained by adding the single variable
$x$ to $F_0$.

An $(F,P)$-model $M$ is a tuple 
\[
M=(|M|,\{ M_\sigma \co |M|^n \ra |M| \mid \sigma\in F_n, n\in\omega \},
\{ M_\pi \subseteq |M|^n \mid \pi\in P_n, n\in\omega \}).
\]
where $|M|$ is a set called the \emph{carrier of $M$}. 
An $(F,P)$-model homomorphism $h\co M \ra N$ is a function $|M| \ra
|N|$ such that $h(M_\sigma (x_1,\dots,x_n)) = N_\sigma
(h(x_1),\dots,h(x_n))$ for any $\sigma \in F_n$ and $h(M_\pi)
\subseteq N_\pi$ for each $\pi\in P_n$. 

The satisfaction relation $M \models^{\FOL}_{(F,P)} \rho$ is the usual
Tarskian style satisfaction defined on induction on the structure of
the sentence $\rho$. 

Given a signature morphism $\varphi \co (F,P) \ra (F',P')$, the
induced sentence translation $\Sen^{\FOL}(\varphi)$ just replaces the
  symbols of any $(F,P)$-sentence with symbols from $(F',P')$
  according $\varphi$, and the induced model reduct
  $\Mod^{\FOL}(\varphi)(M')$ leaves the carrier set as it is and for
  any $x$ function or predicate symbol of $(F,P)$, it interprets $x$ as
  $M'_{\varphi(x)}$.  

In what follows we shall also consider the following parts (or
`sub-institutions') of $\FOL$ that are determined by restricting the
$\FOL$ signatures as follows:
\begin{itemize}\smallitems

\item $\REL$: no function symbols (hence $\Sign^{\REL} \cong
  \Set^\omega$);

\item $\BREL$: no function symbols and only one binary predicate symbol
  $\lambda$   (hence $\Sign^\BREL \cong \{ \lambda \}$);  

\item $\SETC$: no predicate symbols and no function symbols of arity
  greater than $0$ (hence  $\Sign^{\SETC} \cong \Set$); 

\item $\BRELC$: one binary predicate symbol and no function symbols of arity
  greater than $0$ (hence $\Sign^{\BRELC} \cong \Set$); 


\end{itemize}
\end{rm}
\end{example}
\section{Stratified institutions}

The structure and contents of this section is as follows:  
\begin{enumerate}

\item We recall the definition of stratified institution of
  \cite{strat} and slightly upgrade it; 

\item We provide two canonical extractions of ordinary institutions out
  of stratified institutions, corresponding to the local and global
  satisfaction in modal logic, respectively; 

\item We present a series of examples of modal logical systems
  captured as stratified institutions. 

\end{enumerate}

\subsection{Stratified institutions: the concept} 

Informally, the main idea behind the concept of stratified institution
as introduced in \cite{strat} is to enhance the concept of institution
with `states' for the models.
Thus each model $M$ comes equipped with a \emph{set} $\sem{M}$. 
A typical example is given by the Kripke models, where $\sem{M}$ is
the set of the possible worlds in the Kripke structure $M$.

\begin{definition}[Stratified institution]\label{strat-dfn}
A {\em stratified institution} 
$\I = (\Sign^\I, \Sen^\I, \Mod^\I,\sem{\_}^\I,\models^\I)$
consists of: 
\begin{itemize}\smallitems

\item[--] a category $\Sign^\I$ of signatures,

\item[--] a sentence functor $\Sen^\I \co  \Sign^\I \rightarrow \Set$;

\item[--] a model functor $\Mod^\I \co (\Sign^\I)^{\op} \rightarrow \CAT$; 

\item[--] a ``stratification'' lax natural transformation $\sem{\_}^\I\co
  \Mod^\I \Ra \SET$, where $\SET\co \Sign^\I \ra \CAT$ is a functor
  mapping each signature to $\Set$; and 

\item[--] a satisfaction relation between models and sentences which is
parameterized by model states,\linebreak
$M \ (\models^\I)^w_\Sigma \ \rho$ where $w \in \sem{M}^\I_\Sigma$
such that  
\begin{equation}\label{strat-sat-cond-eq}
\Mod^\I(\varphi)(M) \ \ (\models^\I)^{\sem{M}_\varphi (w)}_\Sigma \ \ \rho
\mbox{ \ if and only if \ } 
M \ \ (\models^\I)^w_{\Sigma'} \ \ \Sen^\I(\varphi)(\rho)
\end{equation}
holds for any signature morphism $\varphi \co \Sigma \ra \Sigma'$,
$\Sigma'$-model $M$, $w\in \sem{M}^\I_{\Sigma'}$, and
$\Sigma$-sentence $\rho$. 
\end{itemize}
Like for ordinary institutions, when appropriate we shall also use
simplified notations without superscripts or subscripts that are clear
from the context.  
\end{definition}

The lax natural transformation property of $\sem{\_}$ is depicted in
the diagram below
\[
\xymatrix @C+2em {
\Sigma''  &  
  \Mod(\Sigma'') \ar[r]^{\sem{\_}_{\Sigma''}} \ar[d]_{\Mod(\varphi')}
  \ar@/^.9pc/[dr]_{\quad} 
  \ar@/_.7pc/[dr]
   &
  \Set
  \ar@{}[ld]^(.35){}="a"^(.62){}="b" 
  \ar@{=>} "a";"b"^{ \ \ \sem{\_}_{\varphi'}}  
  \ar[d]^{=}
     \\
\Sigma' \ar[u]_{\varphi'} &  
  \Mod(\Sigma') \ar[d]_{\Mod(\varphi)} \ar[r]|{\sem{\_}_{\Sigma'}}  
  \ar@/^.9pc/[dr]_{\quad} 
  \ar@/_.7pc/[dr] & 
  \Set  \ar[d]^{=} 
  \ar@{}[ld]^(.35){}="c"^(.62){}="d" 
  \ar@{=>} "c";"d"^{ \ \ \sem{\_}_{\varphi}}  
&  \\
\Sigma \ar[u]_{\varphi} &  
  \Mod(\Sigma) \ar[r]_{\sem{\_}_{\Sigma}} & \Set
}
\]
with the following compositionality property for each $\Sigma''$-model
$M''$: 
\[
\sem{M''}_{(\varphi';\varphi)} =
\sem{M''}_{\varphi'};\sem{\Mod(\varphi')(M'')}_\varphi.
\] 
Moreover the natural transformation property of each
$\sem{\_}_\varphi$ is given by the commutativity of the following
diagram:
\begin{equation}\label{diag1} 
\xymatrix{
M' \ar[d]_{h'}  & & 
  \sem{M'}_{\Sigma'} \ar[r]^-{\sem{M'}_{\varphi}}
  \ar[d]_{\sem{h'}_{\Sigma'}} & 
  \sem{\Mod(\varphi)(M')}_{\Sigma} \ar[d]^{\sem{\Mod(\varphi)(h')}_\Sigma}
 \\
N' & & 
 \sem{N'}_{\Sigma'} \ar[r]_-{\sem{N'}_{\varphi}} & 
  \sem{\Mod(\varphi)(N')}_{\Sigma}
}
\end{equation}

\

The satisfaction relation can be presented as a natural transformation
$\models \co \Sen \Ra \sem{\Mod(\_) \ra \Set}$ where the functor 
$\sem{\Mod(\_) \ra \Set} \co \Sign \ra \Set$ is defined by 
\begin{itemize}\smallitems

\item[--] for each signature $\Sigma\in |\Sign|$, 
$\sem{\Mod(\Sigma) \ra \Set}$ denotes the set of all the mappings 
$f \co |\Mod(\Sigma)| \ra \Set$ such that $f(M) \subseteq
\sem{M}_\Sigma$; and 

\item[--] for each signature morphism $\varphi \co \Sigma \ra \Sigma'$, 
$\sem{\Mod(\varphi) \ra \Set}(f)(M') = 
\sem{M'}_\varphi^{-1}(f(\Mod(\varphi)(M')))$.

\end{itemize}
A straightforward check reveals that the Satisfaction Condition
(\ref{strat-sat-cond-eq}) appears exactly as the naturality property
of $\models$: 
$$\xy
\xymatrix{
\Sigma \ar[d]_{\varphi}  & & 
  \Sen(\Sigma) \ar[r]^-{\models_\Sigma}
  \ar[d]_{\Sen(\varphi)} & 
  \sem{\Mod(\Sigma) \ra \Set} \ar[d]^{\sem{\Mod(\varphi) \ra \Set}}
 \\
\Sigma' & & 
  \Sen(\Sigma') \ar[r]_-{\models_{\Sigma'}} & 
  \sem{\Mod(\Sigma') \ra \Set}
}
\endxy$$

Ordinary institutions are the stratified institutions
for which $\sem{M}_\Sigma$ is always a singleton set. 
In Dfn.~\ref{strat-dfn} we have removed the surjectivity
condition on $\sem{M'}_\varphi$ from the definition of the stratified
institutions of \cite{strat} and will rather make it  explicit when
necessary. 
This is motivated by the fact that most of the results developed do
not depend upon this condition which however holds in all examples
known by us.  
In fact in most of the examples $\sem{M'}_\varphi$ are even
identities, which makes $\sem{\_}$ a strict rather than lax natural
transformation.
A notable exception, when $\sem{\_}$ is a proper lax natural
transformation is given by Ex.~\ref{ofol-ex}. 
Also the definition of stratified institution of \cite{strat} did not
introduce $\sem{\_}$ as a lax natural transformation, but rather as an
indexed family of mappings without much compositionality properties,
which was enough for the developments in \cite{strat}.  

The following very expected property does not follow from the axioms
of Dfn.~\ref{strat-dfn}, hence we impose it explicitly. 
It holds in all the examples discussed in this paper. 

\noindent
\textbf{Assumption:} In all considered stratified institutions the
satisfaction is preserved by model isomorphisms, i.e. for each
$\Sigma$-model isomorphism $h \co M \ra N$, each $w\in
\sem{M}_\Sigma$, and each $\Sigma$-sentence $\rho$, 
\[
M \models^w \rho \mbox{ \ if and only if \ }
N \models^{\sem{h}(w)} \rho.
\]

\subsection{Reducing stratified institutions to ordinary institutions}

The following construction will be used systematically in what
follows for reducing stratified institution theoretic concepts to
ordinary institution theoretic concepts, and consequently for
reusing results from the latter to the former realm. 

\begin{fact}\label{sharp-institution-fact}
Each  stratified institution 
$\I = (\Sign,\Sen,\Mod,\sem{\_},\models)$ determines the
following ordinary institution 
$\I^\sharp = (\Sign,\Sen,\Mod^\sharp,\models^\sharp)$ (called the
\emph{local institution of $\I$}) where
\begin{itemize}\smallitems

\item[--] the objects of $\Mod^\sharp (\Sigma)$ are the pairs $(M,w)$
  such that $M\in |\Mod(\Sigma)|$ and $w\in \sem{M}_\Sigma$;  

\item[--] the $\Sigma$-homomorphisms  $(M,w)\ra (N,v)$ are the pairs
  $(h,w)$ such that $h \co M \ra N$ and $\sem{h}_\Sigma (w) = v$; 

\item[--] for any signature morphism $\varphi\co \Sigma \ra \Sigma'$
  and any $\Sigma'$-model $(M',w')$
\[
\Mod^\sharp (\varphi)(M',w') = 
(\Mod(\varphi)(M'),\sem{M'}_\varphi (w'));
\] 

\item[--] for each $\Sigma$-model $M$, each $w\in \sem{M}_\Sigma$, and
  each $\rho\in \Sen(\Sigma)$
\[
((M,w) \models^\sharp_\Sigma \rho) = (M \models^w_\Sigma \rho).
\]
\end{itemize} 
\end{fact}

The preservation of $\models$ under model isomorphisms imply 
the preservation of $\models^\sharp$ under model isomorphisms.
This follows immediately by noting that $(h,w)$ is a model
isomorphism in $\I^\sharp$ if and only if $h$ is a model isomorphism
in $\I$. 

The following second interpretation of  stratified
institutions as ordinary institutions has been given in
\cite{strat}.
Note that unlike $\I^\sharp$ above, $\I^*$ below shares with $\I$ the
model functor. 

\begin{definition}
For any stratified institution $\I =
(\Sign,\Sen,\Mod,\sem{\_},\models)$ we say that \emph{$\sem{\_}$ is
surjective} when for each
signature morphism $\varphi \co \Sigma \ra \Sigma'$ and each
$\Sigma'$-model $M'$, $\sem{M'}_\varphi \co \sem{M'}_{\Sigma'} \ra
\sem{\Mod(\varphi)(M')}_{\Sigma}$ is surjective.  
\end{definition}

\begin{fact}
Each  stratified institution $\I =
(\Sign,\Sen,\Mod,\sem{\_},\models)$ with $\sem{\_}$ surjective
determines an (ordinary) institution $\I^* =
(\Sign,\Sen,\Mod,\models^*)$ (called the \emph{global institution of
  $\I$}) by defining   
\[
(M \models^*_\Sigma \rho) = 
\bigwedge \{ M \models^w_\Sigma \rho \mid w\in \sem{M}_\Sigma \}. 
\]
\end{fact} 

\begin{fact}
Let $\I$ be a  stratified institution $\I$ with $\sem{\_}$ surjective.
For each $E\subseteq \Sen(\Sigma)$ and each $\rho\in \Sen(\Sigma)$, we
have that 
\[
E \models^\sharp \rho \mbox{ \ implies \ } E \models^* \rho.
\]
\end{fact}

The institutions $\I^\sharp$ and $\I^*$ represent generalizations of
the concepts of local and global satisfaction, respectively, from modal
logic (e.g. \cite{blackburn-rijke-venema2002}). 

\subsection{Examples of stratified institutions}\label{ex-sect}

\begin{example}[Modal propositional logic ($\MPL$)]\label{mpl-ex}
\begin{rm}
This is the most common form of modal logic
(e.g. \cite{blackburn-rijke-venema2002}, etc.).
 
Let $\Sign^{\MPL} = \Set$.
For any signature $P$, commonly referred to as `set of propositional
variables', the set of its sentences 
$\Sen^{\MPL}(P)$ is the set $S$ defined by the following grammar
\begin{equation}\label{mpl-grammar-equation}
S \ ::= \ P \mid S \wedge S \mid \neg S \mid \Diamond S
\end{equation}
A $P$-model is Kripke structure $(W,M)$ where 
\begin{itemize}\smallitems

\item $W = (|W|,W_\lambda)$ consists of set (of `possible worlds')
  $|W|$ and an `accesibility' relation $W_\lambda \subseteq |W| \times
  |W|$; and 

\item $M \co |W| \ra 2^P$. 

\end{itemize}
A homomorphism $h \co (W,M) \ra (V,N)$ between Kripke structures is a
homomorphism of binary relations $h \co W \ra V$ (i.e. $h \co |W| \ra
|V|$ such that $h(W_\lambda) \subseteq V_\lambda$) and such that for
each $w\in |W|$, $M^w \subseteq N^{h(w)}$. 

The satisfaction of any $P$-sentence $\rho$ in a Kripke
structure $(W,M)$ at $w\in |W|$ is defined by recursion on the
structure of $\rho$: 
\begin{itemize}\smallitems

\item $((W,M) \models_P^w \pi) = (\pi \in M^w)$;

\item $((W,M) \models_P^w \rho_1 \wedge \rho_2) = ((W,M) \models_P^w
  \rho_1) \wedge ((W,M) \models_P^w \rho_2)$; 

\item $((W,M) \models_P^w \neg \rho) = 
  \neg ((W,M) \models_P^w \rho)$; and 

\item $((W,M) \models_P^w \Diamond \rho) = 
\bigvee_{(w,w')\in W_\lambda}((W,M) \models_P^{w'} \rho)$. 

\end{itemize}
For any function $\varphi \co P \ra P'$ the $\varphi$-translation of a
$P$-sentence just replaces each $\pi\in P$ by $\varphi(\pi)$ and the
$\varphi$-reduct of a $P'$-structure $(W,M')$ is the $P$-structure
$(W,M)$ where for each $w\in |W|$, $M^w = \varphi;M'^w$. 

The stratification is defined by $\sem{(W,M)}_P = |W|$.   

Various `sub-institutions' of $\MPL$ are obtained by restricting the
semantics to particular classes of frames.
Important examples are ${\MPL}t$, ${\MPL}s4$, and ${\MPL}s5$ which are
obtained by restricting the frames $W$ to those which are
respectively, reflexive, preorder, or equivalence (see
e.g. \cite{blackburn-rijke-venema2002}).     
\end{rm}
\end{example} 

\begin{example}[First order modal logic ($\MFOL$)]
\begin{rm}
First order modal logic \cite{fitting-mendelsohn98} extends classical
first order logic with modalities in the same way 
propositional modal logic extends classical propositional logic.
However there are several variants that differ slightly in the
approach of the quantifications. 
Here we present a capture of one of the most common variants of first
order modal logic as a stratified institution. 

$\MFOL$ has the category of signatures of $\FOL$ but for the sentences
adds $S ::= \Diamond S$ to the $\FOL$ grammar (\ref{fol-grammar}). 
The $\MFOL$ $(F,P)$-models upgrade the $\MPL$ Kripke structures
$(W,M)$ to the first order situation by letting $M \co |W| \ra
|\Mod^{\FOL}(F,P)|$ such that the following sharing conditions hold:
for any $i,j \in |W|$, $|M^i| = |M^j|$ and also $M^i_x = M^j_x$ for
each constant $x\in F_0$. 
The concept of $\MFOL$-model homomorphism is also an upgrading of the
concept of $\FOL$-model homomorphism as follows: $h \co (W,M) \ra
(V,N)$ is pair $(h_0,h_1)$ where $h_0 \co W \ra V$ is a homomorphism
of binary relations (like in $\MPL$) and $h_1 \co M^w \ra N^{h_0 (w)}$
is an $(F,P)$-homomorphism of $\FOL$-models for each $w\in |W|$. 

The satisfaction $(W,M) \models^{\MFOL}_{(F,P)} \rho$ is defined by
recursion on the structure of $\rho$, like in $\MPL$ for $\wedge$,
$\neg$, and $\Diamond$, for the atoms the $\FOL$ satisfaction relation
is used, and for the quantifier case $(W,M) \models_{(F,P)} (\exists
x) \rho$ if and only if there is a valuation of $x$ into $|M|$ such
that $(W,M') \models_{(F+x,P)} \rho$ for the corresponding expansion 
$(W,M')$ of $(W,M)$ to $(F\!+\!x,P)$. 
(This makes sense because in any $\MFOL$ Kripke structure the
interpretations of the carriers and of the constants are shared.) 

The translation of sentences and the model reducts corresponding to an
$\MFOL$ signature morphism are obtained by the obvious blend of the
corresponding translations and reducts, respectively, in $\MPL$ and
$\FOL$. 

The stratification is like in $\MPL$, with $\sem{(W,M)}_{(F,P)} = |W|$.   

In the institution theory literature
(e.g. \cite{iimt,ks,HybridIns,QVHybrid}) first order modal logic is
often considered in a more general form in which the symbols that have
shared interpretations are `user defined' rather than being
`predefined' like here. 
In short this means that the signatures exhibit designated symbols
(sorts, function, or predicate) that are `rigid' in the sense that in
a given Kripke structure they share the same interpretations across
the possible worlds. 
For the single reason of making the reading easier we stick here with
a simpler variant that has constants and the single sort being
predefined as rigid. 
\end{rm}
\end{example}

\begin{example}[Hybrid logics ($\HPL$, $\HFOL$)]\label{hpl-ex}
\begin{rm}
Hybrid logics \cite{prior,blackburn2000} refine modal logics by adding
explicit syntax for the possible worlds. 
Our presentation of hybrid logics as stratified institutions is
related to the recent institution theoretic works on hybrid logics
\cite{HybridIns,QVHybrid}. 

The refinement of modal logics to hybrid ones is achieved by adding a
set component ($\Nom$) to the signatures for the so-called `nominals'
and by adding to the respective grammars 
\begin{equation}\label{hybrid-grammar}
S ::= \sen{i} \mid @_i S \mid (\exists i) S'
\end{equation}
where $i\in \Nom$ and $S'$ is the set of the sentences of the
signature that extends $\Nom$ with the nominal variable $i$. 
The models upgrade the respective concepts of Kripke structures to
$(W,M)$ by adding to $W$ interpretations of the nominals, i.e. 
$W = (|W|,\{ W_i \in |W| \mid i\in \Nom \},W_\lambda)$.
The satisfaction relations between models (i.e. Kripke structures) and
sentences extend the satisfaction relations of the corresponding
non-hybrid modal institutions with 
\begin{itemize}\smallitems

\item $((W,M) \models^w \sen{i}) = (W_i = w)$; 

\item $((W,M) \models^w @_i \rho) = ((W,M) \models^{W_i} \rho)$; and 

\item $((W,M) \models^w (\exists i)\rho) = 
\bigvee \{(W',M) \models^w \rho \mid W' \mbox{ \ expansion of \ }W
\mbox{ \ to \ } \Nom\!+\!i \}$. 

\end{itemize}
Note that quantifiers over nominals allow us to simulate the binder
operator $(\downarrow \rho)$ of \cite{goranko96} by 
$(\forall i) i\impl\rho$.  

The translation of sentences and model reducts corresponding to
signature morphisms are canonical extensions of the corresponding
concepts from $\MPL$ and $\MFOL$. 

The stratifications of $\HPL$ and $\HFOL$ are like for $\MPL$ and
$\MFOL$, i.e. $\sem{(W,M)}_{(\Nom,\Sigma)} = |W|$. 
\end{rm}
\end{example}

\begin{example}[Polyadic modalities ($\MMPL$, $\MHPL$, $\MMFOL$,
  $\MHFOL$)] 
\begin{rm}
Multi-modal logics (e.g. \cite{gabbay-mml2003}) exhibit several
modalities instead of only the traditional $\Diamond$ and $\Box$ and
moreover these may have various arities.   
If one considers the sets of modalities to be variable then they have
to be considered as part of the signatures. 
We may extend each of $\MPL$, $\HPL$, $\MFOL$ and $\HFOL$ to the
multi-modal case, 
\begin{itemize}\smallitems

\item by adding an `$\mathcal{M}$' in front of each of
these names;
 
\item by adding a component $\Lambda = (\Lambda_n)_{n\in\omega}$ to
  the respective signature concept (with $\Lambda_n$ standing for the
  modalities symbols of arity $n$), e.g. an $\MHFOL$ signature would
  be a tuple of the form $(\Nom,\Lambda,(F,P))$; 

\item by replacing in the respective
  grammars the rule $S ::= \Diamond S$ by the set of rules 
\[
\{ S ::= \langle \lambda \rangle S^n \mid \lambda\in \Lambda_{n+1}, n\in\omega \};
\]
\item by replacing the binary relation $W_\lambda$ from the models
  $(W,M)$ with a set of interpretations 
  $\{ W_\lambda \subseteq |W|^n \mid \lambda\in\Lambda_n, n\in\omega \}$.   

\end{itemize}
Consequently the definition of the satisfaction relation gets upgraded
with 
\[
\mbox{ \ for each \ }
\lambda\in\Lambda_{n+1}, \  
((W,M) \models^w \langle\lambda\rangle(\rho_1,\dots,\rho_n)) =
\big(\bigvee_{(w,w_1,\dots,w_n)\in W_\lambda}\bigwedge_{1\leq i\leq n}
(W,M)\models^{w_i} \rho_i\big).
\]
The stratification is the same like in the previous examples,
i.e. $\sem{(W,M)}_{(\Nom,\Lambda,\Sigma)} =|W|$. 
\end{rm}
\end{example}

\begin{example}[Modalizations of institutions; $\HHPL$]
\begin{rm}
In a series of works \cite{ks,HybridIns,QVHybrid} modal logic and
Kripke semantics are developed by abstracting away details that do not
belong to modality, such as sorts, functions, predicates, etc. 
This is achieved by extensions of abstract institutions (in the
standard situations meant in principle to encapsulate the atomic part
of the logics) with the essential ingredients of modal logic and
Kripke semantics. 
The result of this process, when instantiated to various concrete
logics (or to their atomic parts only) generate uniformly a wide range
of hierarchical combinations between various flavours of modal logic
and various other logics. 
Concrete examples discussed in \cite{ks,HybridIns,QVHybrid} include
various modal logics over non-conventional structures of relevance in
computing science, such as partial algebra, preordered algebra, etc. 
Various constraints on the respective Kripke models, many of them
having to do with the underlying non-modal structures, have also been
considered. 
All these arise as examples of stratified institutions like the
examples presented above in the paper. 
This great multiplicity of non-conventional modal logics constitute
an important range of applications for this work.

An interesting class of examples that has emerged quite smoothly out
of the general works on hybridization\footnote{I.e. Modalization including
  also hybrid logic features.} of institutions is that of
multi-layered hybrid logics that provide a logical base for specifying
hierarchical transition systems (see \cite{madeira-phd}).  
As a single simple example let us present here the double layered
hybridization of propositional logic, denoted $\HHPL$.\footnote{Other
  interesting examples that may be obtained by double or multiple
  hybridizations of logics would be $\HHFOL$, $\HHHPL$, etc., and also
  their polyadic multi-modalities extensions.}   
This amounts to a hybridization of $\HPL$, its models thus being
``Kripke structures of Kripke structures''.  

The $\HHPL$ signatures are triples $(\Nom^0,\Nom^1,P)$ with
$\Nom^0$ and $\Nom^1$ denoting the nominals of the first and second
layer of hybridization, respectively. 
The $(\Nom^0,\Nom^1,P)$-sentences are built over the two hybridization
layers by taking the $(\Nom^0,P)$-sentences as atoms in the grammar
for the $\HPL$ sentences with nominals from $\Nom^1$.   
In order to prevent potential ambiguities, in general we tag the
symbols of the respective layers of hybridization by the superscripts
$0$ (for the first layer) and $1$ (for the second layer).  
This convention should include nominals and connectives ($\Diamond$,
$\wedge$, etc.) as well as quantifiers. 
For instance, the expression  $@_{j^1} k^0 \wedge^1 \Box^1 \rho$ is a
sentence of $\HHPL$ where the symbols $k$ and $j$ represent nominals
of the first and second level of hybridization and $\rho$ a $\PL$
sentence. 
On the other hand, according to this tagging convention the expression
$@_{j^0} k^1 \wedge^1 \Box^1 \rho$ would not parse. 

Our tagging convention extends also to $\HHPL$ models.
A $(\Nom^0,\Nom^1,P)$-model is a pair $(W^1,M^1)$ with $W^1$ being a
$\Mod^\BRELC (\lambda)$ model and $M^1 = ((M^1)^w)_{w\in |W^1|}$ where
$(M^1)^w$ is a $(\Nom^0,P)$-model in $\HPL$, denoted $((W^0)^w,(M^0)^w)$. 
We also require that for all $w,w'\in |W^1|$, we have that
$|(W^0)^w|=|(W^0)^{w'}|$ and $(W^0)^w_i = (W^0)^{w'}_i$ 
for each $i\in \Nom^0$. 

These definitions extend in the obvious way to signature morphisms,
sentence translations, model reducts and satisfaction relation. 
We leave these details as exercise for the reader. 
Then $\HHPL$ has the same stratified structure like $\HPL$ and
$\HFOL$, namely $\sem{(W^1,M^1)}_{(\Nom^0,\Nom^1,P)} = |W^1|$.

It is easy to see that in $\HHPL$ the semantics of the Boolean
connectors and of the quantifications with nominals of the lower layer
is invariant with respect to the hybridization layer; this means that
in these cases the tagging is not necessary.
For example if $\rho$ is an $\HPL$ sentence then $(\forall^1 i^0)\rho$
and $(\forall^0 i^0)\rho$ are semantically equivalent, while if $\rho$
is not an $\HPL$ sentence (which means it has some ingredients from
the second layer of hybridization) then $(\forall^0 i^0)\rho$ would
not parse. 
In both cases just using the notation $(\forall i^0)$ would not carry
any ambiguities. 
\end{rm}  
\end{example}

The next series of examples include multi-modal first order logics
whose semantics are given by ordinary first order rather than
Kripke structures.   

\begin{example}[Multi-modal open first order logic ($\OFOL$, $\MOFOL$,
  $\HOFOL$, $\HMOFOL$)]\label{ofol-ex}
\begin{rm}
The stratified institution $\OFOL$ is a the $\FOL$ instance of $St(\I)$, the
`internal stratification' abstract example developed in \cite{strat}. 
An $\OFOL$ signature is a pair $(\Sigma,X)$ consisting of $\FOL$
signature $\Sigma$ and a finite block of variables. 
An $\OFOL$ signature morphism $\varphi \co (\Sigma,X) \ra
(\Sigma',X')$ is just a $\FOL$ signature morphism $\varphi \co \Sigma
\ra \Sigma'$ such that $X \subseteq X'$. 

We let $\Sen^\OFOL((F,P),X)=\Sen^\FOL(F+X,P)$ and 
$\Mod^\OFOL ((F,P),X) = \Mod^\FOL(F,P)$. 

For each $((F,P),X)$-model $M$, each $w\in |M|^X$, and each
$((F,P),X)$-sentence $\rho$ we define 
\[
(M (\models^\OFOL_{(F,P),X})^w \rho) = (M^w \models^\FOL_{(F+X,P)} \rho)
\]
where $M^w$ is the expansion of $M$ to $(F\!+\!X,P)$ such that 
$M^w_X = w$. 
This is a stratified institution with $\sem{M}_{\Sigma,X} = |M|^X$
for each $(\Sigma,X)$-model $M$.
For any signature morphism $\varphi \co (\Sigma,X) \ra (\Sigma',X')$
and any $(\Sigma',X')$-model $M'$, $\sem{M'}_\varphi \co |M'|^{X'} \ra
|M'|^X$ is defined by $\sem{M'}_\varphi (a) = a|_X$ (i.e. the
restriction of $a$ to $X$). 
Note that $\sem{M'}_\varphi$ is surjective and that this provides an
example when $\sem{\_}$ is a proper lax natural transformation. 

We may refine $\OFOL$ to a multi-modal logic ($\MOFOL$) by adding 
\[
\{ S ::= \langle\pi\rangle S^n \mid \pi\in P_{n+1}, n\in\omega \}
\]
to the grammar defining each $\Sen^\OFOL ((F,P),X)$ and consequently  
by extending the definition of the satisfaction relation with 
\begin{itemize}\smallitems

\item 
$(M \models^w \langle\pi\rangle (\rho_1,\dots,\rho_n)) =
\bigvee_{(w,w_1,\dots,w_n)\in (M^X)_\pi} \bigwedge_{1\leq i\leq
  n}(M\models^{w_i}\rho_i)$ for each $\pi\in P_{n+1}$, $n\in\omega$.

\end{itemize}
(Here and elsewhere $M^X$ denotes the $X$-power of $M$ in the category
of $\FOL$ $(F,P)$-models.)

Or else we may refine $\OFOL$ with nominals ($\HOFOL$) by adding the
grammar for nominals (\ref{hybrid-grammar}), for each constant $i\in
F_0$, to the grammar defining each $\Sen^\OFOL ((F,P),X)$ and
consequently extending the definition of the satisfaction relation
with   
\begin{itemize}\smallitems

\item $(M\models_{(F,P),X}^w \sen{i}) = ((M^X)_i = w)$;

\item $M \models_{(F,P),X}^w @_i \rho) = 
(M \models_{(F,P),X}^{(M^X)_i} \rho)$;

\item $(M\models_{(F,P),X}^w (\exists i)\rho) = 
\bigvee \{M' \models_{(F\!+\!i,P),X}^w \rho \mid M' \mbox{ \ expansion of \ }M
\mbox{ \ to \ } (F\!+\!i,P) \}$.

\end{itemize}
We can also have $\HMOFOL$ as the blend between $\HOFOL$ and
$\MOFOL$. 
\end{rm}
\end{example}

\section{The logic of stratified institutions}
\label{logic-sect}

We start the section by extending the definition of the semantics of
Boolean connectives and quantifiers from ordinary institutions (see
\cite{tarlecki86s,upins,iimt} etc.) to stratified institutions. 
After this, based on the stratified structure of stratified
institutions, we define the semantics of modalities, nominals, $@$ at
the level of abstract stratified institutions. 
In each of these cases a minimally sufficient additional structure is
employed. 

\begin{definition}\label{int-logic-dfn}
In any stratified institution 
$\I = (\Sign,\Sen,\Mod,\sem{\_},\models)$
\begin{itemize}\smallitems

\item a $\Sigma$-sentence $\rho_1 \!\wedge\! \rho_2$ is an
  \emph{external conjunction} of $\Sigma$-sentences $\rho_1$ and
  $\rho_2$ when for each $\Sigma$-model $M$ and each $w\in
  \sem{M}_\Sigma$,  
\vspace{-.6em}
\[
(M \models^w \rho_1 \!\wedge\! \rho_2) = 
(M \models^w \rho_1) \wedge (M \models^w \rho_2);
\]
\item a $\Sigma$-sentence $\rho_1 \!\impl\! \rho_2$ is an
  \emph{external implication} of $\Sigma$-sentences
  $\rho_1$ and $\rho_2$ when for each $\Sigma$-model $M$ and each
  $w\in \sem{M}_\Sigma$,     
\vspace{-.6em}
\[
(M \models^w \rho_1 \!\impl\! \rho_2) = 
(M \models^w \rho_1) \impl (M \models^w \rho_2);
\]
\item a $\Sigma$-sentence $\rho_1 \!\vee\! \rho_2$ is an
  \emph{external disjunction} of $\Sigma$-sentences  $\rho_1$ and
  $\rho_2$ when for each $\Sigma$-model $M$ and each $w\in \sem{M}_\Sigma$,    
\vspace{-.6em}
\[
(M \models^w \rho_1 \!\vee\! \rho_2) = 
(M \models^w \rho_1) \vee (M \models^w \rho_2);
\]
\item a $\Sigma$-sentence $\neg\rho$ is the \emph{external negation}
  of a $\Sigma$-sentence $\rho$ when for each  $\Sigma$-model $M$
  and each $w\in \sem{M}_\Sigma$,   
\vspace{-.6em}
\[
(M \models^w \neg\rho) = \neg (M \models^w \rho) 
\] 
\item a $\Sigma$-sentence $(\forall \chi)\rho'$ is an 
  \emph{external universal $\chi$-quantification} of a
  $\Sigma'$-sentence $\rho'$ for $\chi \co \Sigma \ra \Sigma'$
  signature morphism when for any 
  $\Sigma$-model $M$ and each $w\in \sem{M}_\Sigma$
\[
(M \models^w_\Sigma (\forall \chi)\rho') = 
\bigwedge_{\Mod(\chi)(M') = M} \big(\bigwedge_{w'\in {\sem{M'}_\chi^{-1}(w)}}
(M'\models^{w'}_{\Sigma'} \rho')\big) 
\]
\item a $\Sigma$-sentence $(\exists \chi)\rho'$ is an 
  \emph{external existential $\chi$-quantification} of a
  $\Sigma'$-sentence $\rho'$ for $\chi \co \Sigma \ra \Sigma'$
  signature morphism when for any 
  $\Sigma$-model $M$ and each $w\in \sem{M}_\Sigma$  
\[
(M \models^w_\Sigma (\exists \chi)\rho') = 
\bigvee_{\Mod(\chi)(M') = M} \big(\bigvee_{w'\in {\sem{M'}_\chi^{-1}(w)}}
(M'\models^{w'}_{\Sigma'} \rho')\big) 
\]
\end{itemize} 
\end{definition}

\begin{remark}
\begin{rm}
In Dfn.~\ref{int-logic-dfn} the notations $\rho_1 \wedge \rho_2$,
$\neg \rho$, etc. are meta-notations in the sense that they may not
correspond to how the actual sentences appear in $\Sen$. 
For example in $\Sen^\MPL(\{ \pi,\pi' \})$ (see Ex.~\ref{mpl-ex}),
according to the respective grammar, there is no actual sentence such
as $\pi \impl\pi'$, however $\MPL$ has implications, in the realm of
the meta notations  $\pi \impl\pi'$ corresponding to the actual
sentence  $\neg (\pi \wedge \neg\pi')$.   
So, these meta-notations of Dfn.~\ref{int-logic-dfn} rather denote
semantical equivalence classes of sentences\footnote{Classes of
  sentences that hold exactly in the same models.}, which goes well
with our work since here we never need to distinguish between
semantically equivalent sentences.
We will keep employing such meta-notations also below in the paper
when introducing the semantics for modalities 
(Dfn.~\ref{mod-strat-institution-dfn}) or for the hybrid features
(Dfn~\ref{hybrid-strat-institution-dfn}).    
\end{rm}
\end{remark}

On the one hand, the concepts of Boolean connectives and
quantifications in ordinary institutions (e.g. from
\cite{tarlecki86a,upins,iimt} etc.) arise as an instance of
Dfn.~\ref{int-logic-dfn} when the underlying set of  each
$\sem{M}_\Sigma$ is a singleton set.  
On the other hand, Fact \ref{logic-sharp-fact} below shows that
Dfn.~\ref{int-logic-dfn} is not a proper generalization of the
corresponding ordinary institution theoretic concepts since
the stratified institution theoretic concepts of Boolean connectives
and quantifications may also be regarded as corresponding instances  
of the respective ordinary institution theoretic concepts.
The importance of Dfn.~\ref{int-logic-dfn} resides thus in the fact
that it gives an explicit account of how Boolean connectors and
quantifications reflect in a stratified setup.   

\begin{fact}\label{logic-sharp-fact}
When they exist, the conjunctions, disjunctions, implications,
negations, universal/existential $\chi$-quantifications coincide in
$\I$ and $\I^\sharp$.  
\end{fact}

\begin{corollary}
In any  stratified institution we have the following:
\begin{enumerate}

\item $\neg(\neg \rho_1 \wedge \neg \rho_2)$ is an external disjunction
  $\rho_1 \!\vee\! \rho_2$; 

\item $\neg \rho_1 \vee \rho_2$ is an external implication $\rho_1
  \!\impl\! \rho_2$; 

\item $\neg(\exists \chi)\neg \rho$ is an external universal
  quantification $(\forall \chi)\rho$. 

\end{enumerate}
\end{corollary}

\begin{proposition}
In any  stratified institution $\I$ with $\sem{\_}$ surjective 
\begin{enumerate}

\item any external conjunctions in $\I$ is an external conjunction
  in $\I^*$ too; and 

\item for any signature morphism $\chi$, any external universal
  $\chi$-quantifications in $\I$ is an external universal
  $\chi$-quantifications in $\I^*$ too.  

\end{enumerate}
\end{proposition}

\begin{proof}
1. For each $\Sigma$-model $M$ and any conjunction $\rho_1 \wedge
\rho_2$ in $\I$ we have that 
$$\begin{array}{rll}
M \models^* \rho_1 \!\wedge\! \rho_2
= & \bigwedge_{w\in \sem{M}} (M \models^w \rho_1 \!\wedge\! \rho_2)
  & (\text{by definition of }\models^*)\\
= & \bigwedge_{w\in \sem{M}} 
    \big((M \models^w \rho_1) \wedge (M \models^w \rho_2)\big)
  & (\text{since }\rho_1\!\wedge\!\rho_2\text{ is conjunction in }\I)\\
= & \big(\bigwedge_{w\in \sem{M}} (M \models^w \rho_1)\big) \wedge 
   \big(\bigwedge_{w\in \sem{M}} (M \models^w \rho_2)\big)
  & \\
= & (M \models^* \rho_1) \wedge (M \models^* \rho_2)
  & (\text{by definition of }\models^*).
\end{array}$$

2. Let $M$ be a $\Sigma$-model and $(\forall \chi)\rho$ a universally
quantified $\Sigma$-sentence in $\I$ for $\chi\co \Sigma \ra \Sigma'$
signature morphism. 
We have that 
\begin{equation}\label{univ-eq1}
\begin{array}{rl}
M \models^*_\Sigma (\forall \chi)\rho =
  & \bigwedge_{w\in \sem{M}} (M \models^w_\Sigma (\forall \chi)\rho)\\
= & \bigwedge \{ (M' \models^{w'}_{\Sigma'} \rho \mid
w\in \sem{M}, \Mod(\chi)(M')=M, w'\in \sem{M'}^{-1}_{\chi} (w) \}. 
\end{array}
\end{equation}
On the other hand we have that 
\begin{equation}\label{univ-eq2}
\bigwedge_{\Mod(\chi)(N')=M} (N' \models^*_{\Sigma'} \rho') = 
\bigwedge_{\Mod(\chi)(N')=M} \big( \bigwedge_{v'\in \sem{N'}_{\Sigma'}}
  (N' \models^{v'}_{\Sigma'} \rho) \big)
\end{equation}
In order to show that $(\forall \chi)\rho$ is an external universal
quantification in $\I^*$ we have to prove that the values in the
equations (\ref{univ-eq1}) and (\ref{univ-eq2}) are equal. 

\fbox{(\ref{univ-eq1}) $\leq$ (\ref{univ-eq2})}
For each $\Mod(\chi)(N')=M$ and $w'\in \sem{N'}_{\Sigma'}$ like in
(\ref{univ-eq2}) we consider $M' = N'$, $w' = v'$ and $w =
\sem{M'}_\chi (w')$ in (\ref{univ-eq1}).  

\fbox{(\ref{univ-eq2}) $\leq$ (\ref{univ-eq1})}
For each $w\in \sem{M}$,  $\Mod(\chi)(M')=M$ and $w'\in
\sem{M'}^{-1}_{\chi} (w)$ like in (\ref{univ-eq1}) we take $N' = M'$
and $v' = w'$ in (\ref{univ-eq2}). 
\end{proof}

In general, $\I^*$ may lack other connectives besides conjunction and
also the existential quantifications that $\I$ does have. 

\begin{definition}[Frame extraction]
Given a stratified institution $\I$, a \emph{frame extraction} is a
pair $L,\Fr$ consisting of a functor $L \co \Sign^\I \ra
\Sign^\REL$ and a lax natural transformation $\Fr \co \Mod^\I \Ra
L;\Mod^\REL$ such that $\sem{\_} = \Fr ; L(\Mod^\REL \Ra \SET)$. 
$$\xy
\xymatrix{
\Mod(\Sigma) \ar[r]^{\sem{\_}_\Sigma} \ar[dr]_{\Fr_\Sigma} & \Set \\
 & \Mod^\REL (L(\Sigma)) \ar[u]_{\text{forgetful}}
}
\endxy$$
\end{definition}

\begin{example}
\begin{rm}
The following table shows some frame extractions for the stratified
institutions introduced above. 

\

\noindent
\begin{tabular}{rll}
stratified institution & $L$ & $\Fr$ \\\hline
$\MPL, \MFOL, \HPL, \HFOL, \HHPL$ & $L(\Sigma) = \{ \lambda \co 2 \}$ &
$\Fr_\Sigma (W,M) = (|W|,W_\lambda)$ \\\hline
$\MMPL, \MMFOL, \MHPL, \MHFOL$ & $L(\Sigma,\Lambda) = \Lambda$ & 
$\Fr_\Sigma (W,M) = (|W|,(W_\lambda)_{\lambda\in
  \Lambda_{n+1},n\in\omega})$ \\\hline
$\MOFOL, \HMOFOL$ & $L((F,P),X) = P$ & $\Fr_\Sigma (M) =
(|M|^X,((M^X)_\pi)_{\pi\in P_{n+1}, n\in \omega})$.  
\end{tabular}
\end{rm}
\end{example}

\begin{definition}\label{mod-strat-institution-dfn}
Let $\I$ be a stratified institution endowed with a frame extraction
$L,\Fr$.  
For any $\lambda \in L(\Sigma)_{n+1}$ and any $\Sigma$-sentences 
$\rho_1, \dots, \rho_n$
\begin{itemize}\smallitems

\item a $\Sigma$-sentence $\langle\lambda\rangle(\rho_1,\dots\rho_n)$ is an
\emph{external $\lambda$-possibility of $\rho_1, \dots, \rho_n$} when
\[
(M \models^w \langle\lambda\rangle(\rho_1,\dots\rho_n)) = 
\bigvee_{(w,w_1,\dots,w_n)\in (\Fr_\Sigma(M))_\lambda} \big(
\bigwedge_{1\leq i \leq n} M \models^{w_i} \rho_i  \big);
\]

\item  a $\Sigma$-sentence $[\lambda](\rho_1,\dots\rho_n)$ is an 
\emph{external $\lambda$-necessity of $\rho_1, \dots, \rho_n$} when
\[
(M \models^w [\lambda](\rho_1,\dots\rho_n)) = 
\bigwedge_{(w,w_1,\dots,w_n)\in (\Fr_\Sigma (M))_\lambda} \big(
\bigvee_{1\leq i \leq n} M \models^{w_i} \rho_i  \big);
\]
\end{itemize}
for each $\Sigma$-model $M$ and for each $w\in \sem{M}_\Sigma$.    
\end{definition}

\begin{fact}
In any  stratified institution like in
Dfn.~\ref{mod-strat-institution-dfn},
$\neg \langle\lambda\rangle (\neg \rho,\dots,\neg\rho_n)$ is a
$\lambda$-necessity of $\rho_1,\dots,\rho_n$. 
\end{fact}

\begin{definition}[Nominals extraction]
Given a stratified institution $\I$, a \emph{nominals extraction} is a
pair $N,\Nm$ consisting of a functor $N \co \Sign^\I \ra
\Sign^\SETC$ and a lax natural transformation $\Nm \co \Mod^\I \Ra
N;\Mod^\SETC$ such that $\sem{\_} = \Nm ; N(\Mod^\SETC \Ra \SET)$. 
$$\xy
\xymatrix{
\Mod(\Sigma) \ar[r]^{\sem{\_}_\Sigma} \ar[dr]_{\Nm_\Sigma} & \Set \\
 & \Mod^\SETC (N(\Sigma)) \ar[u]_{\text{forgetful}}
}
\endxy$$
\end{definition}

\begin{example}
\begin{rm}
The following table shows some nominals extractions for the stratified
institutions introduced above.  
Note that $\HHPL$ admits two such nominals extractions. 

\

\noindent
\begin{tabular}{rll}
stratified institution & $N$ & $\Nm$ \\\hline
$\HPL, \HFOL, \MHPL, \MHFOL$ & $N(\Nom,\Sigma) = \Nom$ &
$\Nm_{(\Nom,\Sigma)} (W,M) = (|W|,(W_i)_{i\in \Nom})$ \\\hline
$\HHPL$ & $N(\Nom^0,\Nom^1,P) = \Nom^0$ & $\Nm (W^1,M^1) =
(|(W^0)^w|,((W^0)^w_i)_{i\in \Nom^0})$ \\ 
        & $N(\Nom^0,\Nom^1,P) = \Nom^1$ & $\Nm (W^1,M^1) =
        (|W^1|,(W^1_i)_{i\in\Nom^1})$ \\\hline 
$\HOFOL, \HMOFOL$ & $N((F,P),X) = F_0$ & 
         $\Nm (M) = (|M|^X, ((M^X)_i)_{i\in F_0})$ \\
\end{tabular}
\end{rm}
\end{example}

\begin{definition}\label{hybrid-strat-institution-dfn}
Let $\I$ be a stratified institution endowed with a nominals
extraction $N, \Nm$. 
For any $i\in \NNom(\Sigma)$
\begin{itemize}\smallitems

\item a $\Sigma$-sentence $\sen{i}$ is an \emph{$i$-sentence} when 
\[
(M \models^w \sen{i}) = ((\Nm_\Sigma (M))_i = w);
\]

\item for any $\Sigma$-sentence $\rho$, a $\Sigma$-sentence $@_i
  \rho$ is the \emph{satisfaction of $\rho$ at $i$} when 
\[
(M \models^w @_i \rho) = (M \models^{(\Nm_\Sigma (M))_i} \rho);
\]
\end{itemize}
for each $\Sigma$-model $M$ and for each $w\in \sem{M}_\Sigma$. 
\end{definition}

\begin{example}
\begin{rm}
The following table shows what of the properties of
Dfn.~\ref{int-logic-dfn}, \ref{mod-strat-institution-dfn} and
\ref{hybrid-strat-institution-dfn} are satisfied by the examples of
stratified institutions given above in the paper. 

\

\noindent
\begin{tabular}{rcccccccccc}
 & $\wedge$ & $\vee$ & $\neg$ & $\impl$ & $(\forall \chi)$ 
   & $(\exists \chi)$ & $\langle\lambda\rangle$ & $[\lambda]$ 
   & $\sen{i}$ & $@_i$ \\\hline 
$\MPL$ & \checkmark & \checkmark & \checkmark & \checkmark & & &
     $\Diamond$ & $\Box$ & & \\\hline
$\MFOL$ & \checkmark & \checkmark & \checkmark & \checkmark &
   $(\forall x)$ & $(\exists x)$ & $\Diamond$ & $\Box$ & & \\\hline
$\HPL$  & \checkmark & \checkmark & \checkmark & \checkmark &
   $(\forall i)$ & $(\exists i)$ & $\Diamond$ & $\Box$ & \checkmark &
   \checkmark \\\hline 
$\HFOL$  & \checkmark & \checkmark & \checkmark & \checkmark &
   $(\forall x)$, $(\forall i)$ & $(\exists x)$, $(\exists i)$ &
   $\Diamond$ & $\Box$ & \checkmark & \checkmark \\\hline 
$\MMPL$ & \checkmark & \checkmark & \checkmark & \checkmark & & &
    \checkmark & \checkmark & & \\\hline 
$\MHPL$ & \checkmark & \checkmark & \checkmark & \checkmark &
    $(\forall i)$ & $(\exists i)$ &
    \checkmark & \checkmark & \checkmark & \checkmark \\\hline 
$\MMFOL$ & \checkmark & \checkmark & \checkmark & \checkmark &
    $(\forall x)$ & $(\exists x)$ & \checkmark & \checkmark & &
    \\\hline 
$\MHFOL$ & \checkmark & \checkmark & \checkmark & \checkmark &
    $(\forall x)$, $(\forall i)$  & $(\exists x)$, $(\exists i)$ &
    \checkmark & \checkmark & \checkmark & \checkmark \\\hline 
$\HHPL$  & \checkmark & \checkmark & \checkmark & \checkmark &
   $(\forall i^0)$, $(\forall i^1)$
   & $(\exists i^0)$, $(\exists i^1)$
   & $\Diamond$ & $\Box$ & $\sen{i^0}$, $\sen{i^1}$ &
   $@_{i^0}$, $@_{i^1}$ \\\hline 
$\OFOL$ & \checkmark & \checkmark & \checkmark & \checkmark &
    $(\forall x)$ & $(\exists x)$ & & & & \\\hline
$\MOFOL$ & \checkmark & \checkmark & \checkmark & \checkmark &
    $(\forall x)$ & $(\exists x)$ & \checkmark & \checkmark & &
    \\\hline 
$\HOFOL$ & \checkmark & \checkmark & \checkmark & \checkmark &
    $(\forall x)$, $(\forall i)$ & $(\exists x)$, $(\exists i)$ & & &
    \checkmark & \checkmark \\\hline 
$\HMOFOL$ & \checkmark & \checkmark & \checkmark & \checkmark &
    $(\forall x)$, $(\forall i)$ & $(\exists x)$, $(\exists i)$ &
    \checkmark & \checkmark & \checkmark & \checkmark  
\end{tabular}

\

\

\noindent
In the table $(\forall x)$, $(\forall i)$ stand for $(\forall \chi)$
where $\chi$ is an extension of the signature with a first order
variable, or a nominal variable, respectively, and similarly for the
existential quantifiers.
The case of the quantifiers reminds us once more that in spite of
the abstract simplicity of the institution theoretic approach to
quantifiers, just based upon model reducts, they are a very powerful
concept supporting a wide range of quantifications within a single
uniform definition. 
Basically, one may quantify over any syntactic entity that is
supported by the respective concept of signature morphisms.
In our examples this means first order variables and nominals alike.
An particularly interesting situation is given by $\HHPL$, where the
concept of signature supports quantification over two kinds of
nominals, corresponding to the two layers of hybridization.  
\end{rm}
\end{example}

\section{Model ultraproducts in stratified institutions}

The structure of the section is as follows:
\begin{enumerate}

\item We start with a recollection of the concept of filtered product
  in abstract categories.

\item Then we discuss filtered products of models in stratified
  institutions and develop some technical results about the
  representation of filtered products of models in $\I^\sharp$, the
  local institution associated to a stratified institution $\I$. 

\item The last part of this section is concerned with the development
  of a {\L}o\'{s} styled theorem for abstract stratified institutions
  that carry some implicit modal structure. 
  This means a gathering of relevant preservation properties for the
  connectives commonly used in sentences in various modal logic
  systems; the connectives are considered by their semantic
  definitions given in Sect.~\ref{logic-sect}. 
  Here also the compactness consequence of {\L}o\'{s} theorem is
  studied both at the level of abstract structured institutions and at
  the level of concrete examples. 

\end{enumerate}

\subsection{A reminder of categorical filtered products}

For each non-empty set $I$ we denote the set of all subsets of $I$ by
$\mathcal{P}(I)$.
A \emph{filter $F$ over $I$}\index{filter} is defined to be a set $F
\subseteq \mathcal{P}(I)$ such that
\begin{itemize} \smallitems

\item $I \in F$,

\item $X \cap Y \in F$ if $X \in F$ and $Y \in F$, and

\item $Y \in F$ if $X \subseteq Y$ and $X \in F$.

\end{itemize}
A filter $F$ is \emph{proper} when $F$ is not $\mathcal{P}(I)$ and it
is an \emph{ultrafilter}\index{ultrafilter} when $X \in F$ if and only
if $(I \setminus X) \not\in F$ for each $X\in \mathcal{P}(I)$.
Notice that ultrafilters are proper filters.
We will always assume that all our filters are proper.

Let $F$ be a filter over $I$ and $I' \subseteq I$.
The \emph{reduction of $F$ to $I'$}\index{reduction of filter}
is denoted by $F|_{I'}$ and defined as $\{ I' \cap X \mid X \in F \}$.

\begin{fact}
The reduction of any filter is still a filter.
\end{fact}

\begin{definition}
A class $\F$ of filters is \emph{closed under reductions}
if and only if $F|_J \in \F$ for each $F\in \F$ and $J\in F$.
\end{definition}
Examples of classes of filters closed under reductions include the
class of all filters, the class of all ultrafilters, the class 
$\{ \{ I \} \mid I \mbox{ \ set}\}$, etc.

\begin{definition}[Categorical filtered products]
\label{cat-filt-prod-dfn}
Let $F$ be a filter over $I$ and $(M_i)_{i\in I}$ a family of objects
in a category with small direct products. 
Then an \emph{$F$-filtered product of $(M_i)_{i\in I}$} 
(or \emph{$F$-product}, for short) 
is a co-limit $\{ \mu_J \co M_J \ra M_F \mid J \in F \}$ of the
directed diagram of canonical projections  
$\{ p_{J \supseteq J'} \co M_J \ra M_{J'} \mid J' \subseteq J \in F
\}$, where for each $J\in F$, $\{ p_{J,i} \co M_J \ra M_i \mid i\in J
\}$ is a direct product of $(M_j)_{j\in J}$.   
$$\xy
\xymatrix @C+2em {
 & M_J \ar[dl]_{p_{J,i}} \ar[d]^{p_{J \supseteq J'}} \ar@/^/[dr]^{\mu_J} &
 \\
M_i & M_{J'}\ar[l]^{p_{J',i}} \ar[r]_{\mu_{J'}} & M_F
}
\endxy$$
If $F$ is an ultrafilter then $F$-products are called
\emph{ultraproducts}\index{ultraproduct of models, in institution}.
\end{definition}
Note that a direct product $\prod_{i\in I} A_i$ is the same as
an $\{ I \}$-product of $(A_i)_{i\in I}$.
Obviously, as co-limits of diagrams of products, filtered products are
unique up to isomorphisms.
Since the co-limits defining filtered products are directed, a
sufficient condition for the existence of filtered 
products, which applies to many situations, is the existence of
small products and of directed co-limits of models.
Note however that this is not a necessary condition because only
co-limits over diagrams of projections are involved.
For example models of higher order logic \cite{codescu-msc,iimt} in
general are known to have only direct products and ultraproducts. 

\begin{definition}[Preservation/lifting of filtered products
  \cite{upins,iimt}] 
Consider a functor $G \co \C' \ra \C$  and $F$ a filter over a set
$I$.
\begin{itemize}\smallitems

\item $G$ \emph{preserves $F$-products} when for each $F$-product
  $\mu'$ of a family $(M'_i)_{i\in I}$ in $|\C'|$, $G(\mu')$ is an
  $F$-product (in $\C$) of   $(G(M'_i))_{i\in I}$. 

\item $G$ \emph{lifts $F$-products} when for each family
  $(M'_i)_{i\in I}$ in $|\C'|$ and each $F$-product $\mu$ in $\C$ of
  $(G(M'_i))_{i\in I}$, there exists an $F$-product $\mu'$ of 
  $(M'_i)_{i\in I}$ in $\C'$ such that $G(\mu') = \mu$. 

\end{itemize}
For any class $\F$ of filters, we say that a functor
\emph{preserves/lifts $\F$-products}
if it preserves/lifts all $F$-products for each filter
$F \in \F$.
\end{definition}

\begin{fact}
If $G$ lifts $F$-products then it also preserves them. 
\end{fact}

In many situations the following applies. 

\begin{fact}
A functor $G$ preserves/lifts $F$-products if it preserves/lifts 
direct products and directed co-limits. 
\end{fact}


The concept has been introduced first time in \cite{upins} under a
different terminology and in a slightly different form, and has been
subsequently used in several works most notably in \cite{ks,iimt}.  

\begin{definition}[Inventing of filtered products]\label{lift-fprod-dfn}
Let  $\F$ be a class of filters closed under reductions.
A functor $G \co \C' \ra \C$ \emph{invents $\F$-products}
\index{lifting of  filtered products by functor} when for each $F\in \F$,
for each $F$-product $\{ \mu_J \co M_J \ra M_F \mid J \in F\}$ of a
family $(M_i)_{i\in I}$  in $|\C|$, and for each $B\in|\C'|$ such that  
$G(B) = M_F$, 
\begin{itemize} \smallitems

\item[--] there exists $J\in F$ and $( M'_i )_{i\in J}$ a family in
  $|\C'|$ such that $G(M'_i) = M_i$ for each $i\in J$ and such that 

\item [--] there exists an $F|_J$-product 
  $\{ \mu'_{J'} \co M'_{J'} \ra B \mid J' \in F|_J \}$ of 
  $(M'_i)_{i\in J}$ such that $G(\mu'_{J'}) = \mu_{J'}$ for each $J'
  \in F|_J$.

\end{itemize}
When $J = I$ we say that $G$
\emph{lifts completely}\index{complete lifting of filtered products by
functor}
the respective $F$-product.
(Note that in this case the closure of $\F$ under reductions is
redundant.)
\end{definition}

In essence, the inventing property of Dfn.~\ref{lift-fprod-dfn} means
that each $\F$-product construction of $G(B)$ can be established as
the image by $G$  of an $\F$-product construction of $B$ by means of a
filter reduction. 

\subsection{Filtered products in stratified institutions}

\begin{definition}
Let $\F$ be any class of filters.
A  stratified institution \emph{has (concrete) $\F$-products} when for
each signature $\Sigma$, $\Mod(\Sigma)$ has $\F$-products (and 
$\sem{\_}_\Sigma \co \Mod(\Sigma) \ra \Set$ preserves $\F$-products).  
\end{definition}
As the following examples show, in practice it is common that
the $\F$-products are concrete.

\begin{example}\label{f-prod-ex}
\begin{rm}
In all examples of Sect.~\ref{ex-sect} the respective stratified
institutions have all $F$-products, which are concrete, as
follows.  
\begin{enumerate}

\item The $F$-products in $\MPL$, $\MFOL$, $\HPL$, $\HFOL$, $\HHPL$
  are obtained as direct instances of the general result on existence
  of $F$-products developed in \cite{ks}.
  In the case of $\HHPL$ this has to be applied twice, first for
  getting $F$-products in $\HPL$ from the $F$-products in $\PL$, and
  then for getting the $F$-products in $\HHPL$ from the $F$-products
  in $\HPL$. 

\item In the case of $\MMPL$, $\MHPL$, $\MMFOL$, $\MHFOL$ we may apply
  a straightforward extension of the above mentioned result of
  \cite{ks} to the multi-modal situation. 

\item In the case of $\OFOL$, $\MOFOL$, $\HOFOL$, $\HMOFOL$ the
  $F$-products are much simpler than in the previous cases
  because the models in all these institutions are just $\FOL$
  models.   

\end{enumerate}
In the case of $\MPL$, $\MFOL$, $\HPL$, $\HFOL$, $\MMPL$, $\MHPL$,
$\MMFOL$, $\MHFOL$, $\HHPL$, according to \cite{ks} the construction
of filtered products is done in two steps, first at the level of 
the Kripke frames and next lifted to the level of the Kripke models in
$\Mod(\Sigma)$; this shows that $\sem{\_}_\Sigma$ creates filtered
products.  
For example, in $\MFOL$ an $F$-product of a family 
$(W_i,M_i)_{i\in I}$ is 
$\{ \mu_J \co (W_J,M_J) \ra (W_F,M_F)\mid J\in F \}$ where 
\begin{itemize}\smallitems

\item $\{ (\mu_J)_0 \co W_J \ra W_F \mid J \in F \}$ is an $F$-product
  of the family of $\BREL$ models $(W_i)_{i\in I}$ where $W_J$ is
  the cartezian product of $(W_i)_{i\in J}$; and 

\item for each $(w_i)_{i\in I} \in |W_I|$ and each $J\in F$ we let 
$M_J^{(w_j)_{j\in J}}$ denote the cartezian product of
$(M_j^{w_j})_{j\in J}$; note that both $|M_J^{(w_j)_{j\in J}}|$ and 
$(M_J^{(w_j)_{j\in J}})_x$ for $x$ constant are invariant with respect
to $(w_i)_{i\in I}$; 

\item let $\{ (\mu_J)_1 \co |M_J^{(w_j)_{j\in J}}| \ra |M_F| \mid J\in
  F \}$ be a directed co-limit in $\Set$; 

\item since the underlying carrier functor 
$|\_| \co \Mod^{\FOL}(\Sigma) \ra \Set$ creates directed co-limits,
for each $(w_i)_{i\in I} \in |W_I|$ we lift the directed co-limit of
the previous item to a directed co-limit  
$\{ (\mu_J)_1 \co M_J^{(w_j)_{j\in J}} \ra 
M_F^{(\mu_I)_0 ((w_i)_{i\in I})} \mid J\in F \}$ 
of $\Mod^{\FOL}(\Sigma)$-models; it is not difficult to check that
the definition of $M_F$ is correct in the sense that 
$(\mu_I)_0 ((w_i)_{i\in I}) = (\mu_I)_0 ((v_i)_{i\in I})$ implies that
$M_F^{(\mu_I)_0 ((w_i)_{i\in I})} = M_F^{(\mu_I)_0 ((v_i)_{i\in I})}$. 

\end{itemize}
In the case of $\OFOL$, $\MOFOL$, $\HOFOL$, $\HMOFOL$,
$\sem{\_}_\Sigma$ is just the composition between a $\FOL$ underlying
carrier functor $M \mapsto |M|$, and a power functor $|M| \mapsto
|M|^X$, which are known (e.g. \cite{iimt}, etc.) to create direct
products and directed co-limits, and thus filtered products.   
\end{rm}
\end{example} 

The following result gives a representation of $F$-products in the
local institution $\I^\sharp$ from the $F$-products in the stratified
institution $\I$. 

\begin{proposition}\label{filt-prod-sharp-prop}
If a stratified institution $\I$ has concrete
$F$-products, then $\I^\sharp$ has $F$-products, which for any  
family 
$\{ (M_i,w_i) \mid M_i \in |\Mod(\Sigma)|, w_i \in \sem{M_i}_\Sigma,
i\in I \}$ may be defined by
\begin{equation}\label{filt-prod-sharp-eqq}
\{ (\mu_J,w_J) \co (M_J,w_J) \ra (M_F, \sem{\mu_I}(w_I)) \mid J \in F \},
\end{equation}
where $\{ \mu_J \co M_J \ra M_F \mid J\in F \}$ is an $F$-product
in $\Mod(\Sigma)$ and $w_J$ is the unique element of $\sem{M_J}$ such
that for each $i\in J$, $\sem{p_{J,i}}(w_J) = w_i$. 
\end{proposition}

\begin{proof}
Let $(M_i)_{i\in I}$ be a family in $|\Mod(\Sigma)|$ and $F$ be a 
filter over $I$. 
We first show that for each $J\in F$,    
\begin{equation}\label{dir-prod-eqq}
\{ (p_{J,i}, w_J)\co (M_J, w_J)\ra (M_i,w_i) \mid i\in J \}
\end{equation}
is a direct product in $\Mod^\sharp(\Sigma)$. 
By the definition of $w_J$, each $(p_{J,i},w_J)$ is well defined, i.e. 
$\sem{p_{J,i}} (w_J) = w_i$. 

For any family $\{ (f_i,v) \co (N,v) \ra (M_i, w_i) \mid i\in J \}$,
by the universal property of the direct products in $\Mod(\Sigma)$
there exists an unique $f \co N \ra M_J$ such that for
each $i\in J$, $f ;p_{J,i} = f_i$. 
$$\xy
\xymatrix @C+1em{
(M_J, w_J) \ar[d]_{(p_{J,i},w_J)} & (N,v) \ar[l]_-{(f,v)}
\ar[dl]^{(f_i,v)}\\ 
(M_i,w_i) & \\
}
\endxy$$
Hence, for each $i\in J$, 
$\sem{p_{J,i}}(\sem{f}(v)) = \sem{f_i}(v) = w_i$.
Since $\sem{p_{J,i}}$ are cartezian projections, it follows that
$\sem{f}(v)= w_J$.   
This completes the proof of the universal property of the direct
product (\ref{dir-prod-eqq}). 

It follows immediately that for each $J' \subset J \in F$, 
$(p_{J \supseteq J'},w_J) \co (M_J, w_J) \ra (M_{J'}, w_{J'})$ is a
corresponding canonical projection in $\Mod^\sharp(\Sigma)$.  
Let us show that (\ref{filt-prod-sharp-eqq}) is a co-limit in
$\Mod^\sharp(\Sigma)$. 
$$\xy
\xymatrix @C+1em {
 & (M_J,w_J) \ar[dl]_{(p_{J\supseteq J'},w_J)} \ar[d]^{(\mu_J,w_J)} 
  \ar@/^2em/[ddr]^{(\nu_J,w_J)} & \\
(M_{J'},w_{J'}) \ar[r]_-{(\mu_{J'},w_{J'})}
\ar@/_2em/[drr]_{(\nu_{J'},w_{J'})} &
(M_F,\sem{\mu_I}(w_I)) \ar[dr]_-{(f,\sem{\mu_I}(w_I))} & \\
 & & (N,v) 
}
\endxy$$
First, note that each $(\mu_J,w_J)$ is well defined, i.e. that 
$\sem{\mu_J}(w_J) = \sem{\mu_I}(w_I)$, which is given by the following
calculation: 
\[
\sem{\mu_I}(w_I) = \sem{p_{I \supset J};\mu_J}(w_I) = 
\sem{\mu_J}(\sem{p_{I \supset J}}(w_I)) = \sem{\mu_J}(w_J).
\]
For establishing the universal property of the co-cone
$(\mu_J,w_J)_{J\in F}$ let us consider another co-cone
$(\nu_J,w_J)_{J\in F}$ over 
$(p_{J \supset J'},w_J)_{J \supset J' \in F}$.
Let $(N,v)$ denote it vertex.
By the universal property of $(\mu_J)_{J\in F}$ in $\Mod(\Sigma)$
there exists an unique $f \co M_F \ra N$ such that for each $J\in F$,
$\mu_J ; f = \nu_J$. 
The argument is completed if we showed that 
$\sem{f}(\sem{\mu_I}(w_I)) = v$.
This holds by the following calculation:
\[
\sem{f}(\sem{\mu_I}(w_I)) = \sem{\mu_I ; f}(w_I) = \sem{\nu_I}(w_I) =
v. 
\]
\end{proof}

\begin{corollary}\label{chi-pres-cor}
For any signature morphism $\chi$ in any  stratified
institution $\I$ with concrete $F$-products, if $\Mod(\chi)$ preserves
$F$-products in $\I$ then $\Mod^\sharp (\chi)$ preserves $F$-products
in $\I^\sharp$.     
\end{corollary}

\begin{proof}
Let $\chi \co \Sigma \ra \Sigma'$ be signature morphism such that
$\Mod(\chi)$ preserves $F$-products and let 
\[
\{ (\mu'_J,w_J) \co (M'_J,w_J) \ra
(M'_F,\sem{\mu'_I}(w_I)) \mid J \in F \}
\] 
be an $F$-product in $\Mod^\sharp (\Sigma')$ like in
Prop.~\ref{filt-prod-sharp-prop}. 
We denote $\Mod(\chi)(M'_i) = M_i$, $\Mod(\chi)(M'_J) = M_J$,
$\Mod(\chi)(M'_F) = M_F$, and $\Mod(\chi)(\mu'_J) = \mu_J$.  
We have to show that 
\[
\{ (\mu_J,\sem{M'_J}_\chi (w_J)) \co 
(M_J,\sem{M'_J}_\chi (w_J)) \ra
(M_F,\sem{M'_F}_\chi (\sem{\mu'_I}(w_I))) \mid J \in F \}
\]
is an $F$-product in $\Mod^\sharp (\Sigma)$.  
First we should establish that for each $J\in F$
\begin{equation}\label{prod-eqq}
\{ (\Mod(\chi)(p_{J,i}),\sem{M'_J}_\chi (w_J)) \co 
(M_J,\sem{M'_J}_\chi (w_J)) \ra 
(M_i,\sem{M'_i}_\chi (w_i)) \mid i\in J \}
\end{equation}
is a direct product.
Consider 
\[
\{ (f_i,v) \co (N,v) \ra (M_i,\sem{M'_i}_\chi (w_i)) \mid i\in J \}.
\]
Since $\Mod(\chi)$ preserves products in $\Mod(\Sigma)$, we have that
the $\I$ part of (\ref{prod-eqq}) is a direct product, hence let $f
\co N \ra M_i$ such that $f;\Mod(\chi)(p_{J,i}) = f_i$.  
For showing that (\ref{prod-eqq}) is a direct product in $\Mod^\sharp
(\Sigma)$ it remains to show that $\sem{f}_\Sigma (v) =
\sem{M'_J}_\chi (w_J)$. 
This holds by the following calculation 
$$\begin{array}{rll}
\sem{\Mod(\chi)(p_{J,i})}_\Sigma (\sem{f}_\Sigma (v)) = & 
  \sem{f_i}(v) & \\
 = & \sem{M'_i}_\chi (w_i) & (\text{by the definition of }f_i) \\
 = & \sem{M'_i}_\chi (\sem{p_{J,i}}_{\Sigma'}(w_J))
   & (\text{by the definition of }w_J) \\
 = & \sem{\Mod(\chi)(p_{J,i})}_\Sigma (\sem{M'_J}_\chi (w_J))
   & (\text{by (\ref{diag1})})
\end{array}$$
and by the fact that $\Mod(\chi)$ and $\sem{\_}_\Sigma$ preserve
direct products, we hav that $\sem{\Mod(\chi)(p_{J,i})}_\Sigma$ are
direct product projections.   

Then it follows immediately that 
$\{ (\Mod(\chi)(p_{J\supseteq J'}),\sem{M'_J}_\chi (w_J)) \mid J'
\subseteq J \in F \}$ is a diagram of projections. 
Now consider any co-cone for the above diagram as follows:
\[
\{ (\nu_J,\sem{M'_J}_\chi (w_J)) \co (M_J,\sem{M'_J}_\chi (w_J)) \ra
(N,v) \mid J \in F \}. 
\]
Since $\Mod(\chi)$ preserves $F$-products it follows that 
$\{ \mu_J \co M_J \ra M_F \mid J \in F \}$ is an $F$-product in
$\Mod(\Sigma)$, hence there exists an unique $f\co M_F \ra N$ such
that for each $J\in F$, $\mu_J ; f = \nu_J$.  
In order to show that $(f,\sem{M'_F}_\chi (\sem{\mu'_I}(w_I)))$ is a
$\Mod^\sharp (\Sigma)$ homomorphism 
$(M_F,\sem{M'_F}_\chi (\sem{\mu'_I}(w_I))) \ra (N,v)$ we still have to
show that  
$\sem{f}_\Sigma (\sem{M'_F}_\chi (\sem{\mu'_I}_{\Sigma'}(w_I))) = v$.
This holds by the following calculation:
$$\begin{array}{rll}
\sem{f}_\Sigma (\sem{M'_F}_\chi (\sem{\mu'_I}_{\Sigma'})) = 
  & \sem{f}_\Sigma (\sem{\mu_I}_\Sigma (\sem{M'_I}_\chi (w_I)))
  & (\text{by (\ref{diag1})}) \\
 = & \sem{\nu_I}_\Sigma ((\sem{M'_I}_\chi (w_I))
   & (\text{since } \mu_I ; f = \nu_I) \\
 = & v & (\text{by the homomorphism property of
 }(\nu_I,(\sem{M'_I}_\chi (w_I))). 
\end{array}$$
\end{proof}

\subsection{{\L}o\v{s} theorem in stratified institutions}

The following definition generalizes the corresponding modal
preservation concept of \cite{ks,iimt} to the much more general
setup of stratified institutions.   

\begin{definition}\label{sen-pres-dfn}
Let $\F$ be a class of filters and let $\I$ be a 
stratified institution with $\F$-products. 
A $\Sigma$-sentence $\rho$ is
\begin{itemize} \smallitems

\item \emph{preserved by $\F$-products}
  when for each   $w\in \sem{M_F}$,
  ``there exists $J\in F$ and $k\in \sem{\mu_J}^{-1} (w)$ such that
  $M_j \models^{k_j} \rho$ for each $j\in J$'' implies
  $M_F \models^w \rho$, and

\item \emph{preserved by $\F$-factors}
  when for each $w\in \sem{M_F}$,
  $M_F \models^w \rho$ implies
  ``there exists $J\in F$ and $k\in \sem{\mu_J}^{-1} (w)$ such that
  $M_j \models^{k_j} \rho$ for each $j\in J$''

\end{itemize}
for each filter $F\in \F$ over a set $I$ and for each family
$( M_j )_{j\in I}$ of $\Sigma$-models, and where 
$\{ \mu_J \co M_J \ra M_F \mid J \in F \}$ denotes an $F$-product of 
$( M_j )_{j\in I}$ and $k_j = \sem{p_{J,j}}_\Sigma (k)$. 
\end{definition}

When all $\sem{M}_\Sigma$ have singletons as their underlying sets,
Dfn.~\ref{sen-pres-dfn} yields the preservation by
$\F$-products/factors in ordinary institutions as defined in
\cite{upins,iimt}. 
On the other hand, the following result shows that stratified
preservations by $\F$-products/factors of Dfn.~\ref{sen-pres-dfn} may
be an instance of their ordinary versions from \cite{upins,iimt}.  

\begin{proposition}\label{pres-sharp-prop}
For any  stratified institution $\I$ with concrete
$\F$-products the following are equivalent for any $\Sigma$-sentence
$\rho$: 
\begin{enumerate}

\item $\rho$ is preserved by $\F$-products/factors in $\I$; and 

\item $\rho$ is preserved by $\F$-products/factors in $\I^\sharp$. 

\end{enumerate}
\end{proposition}

\begin{proof}
In this proof we use the notations of
Prop.~\ref{filt-prod-sharp-prop}.  
First note that since $\I$ has $\F$-products, by
Prop.~\ref{filt-prod-sharp-prop} $\I^\sharp$ has $\F$-products too.
Moreover, by the assumption of preservation of satisfaction by model
isomorphisms, without any loss of generality, we may consider only the
$F$-products given by (\ref{filt-prod-sharp-eqq}) of
Prop.~\ref{filt-prod-sharp-prop}. 

\fbox{$1. \impl 2.$}
For the preservation by $\F$-products, let $(M_i,w_i)_{i\in I}$ and $F\in
\F$ filter over $I$ and assume that there exists $J \in F$ such that
for each $j\in J$, $(M_j,w_j)\models^\sharp \rho$. 
By the definition of $\models^\sharp$ we have that for each $j\in J$,
$M_j \models^{w_j} \rho$.
By 1. it follows that $M_F \models^{\sem{\mu_J} (w_J)}
\rho$. 
Since $\sem{\mu_I}(w_I) = \sem{\mu_J}(\sem{p_{I \supseteq J}}(w_I)) = 
\sem{\mu_J}(w_J)$ it follows that $(M_F,\sem{\mu_I}(w_I)) \models
\rho$. 

For the preservation by $\F$-factors, let $(M_i,w_i)_{i\in I}$ and 
$F\in \F$ filter over $I$ such that $(M_F,\sem{\mu_I}(w_I))
\models^\sharp \rho$.  
Hence $M_F \models^{w'} \rho$ where $w' = \sem{\mu_I}(w_I)$. 
By the hypothesis 1. there exists $J \in F$ and $k\in
\sem{\mu_J}^{-1}(w')$ such that for each $j\in J$, $M_j \models^{k_j}
\rho$. 
Because $\sem{\mu_J} (k) = \sem{\mu_J} (w_J)$ we have that there
exists $J \supseteq J' \in F$ such that $\sem{p_{J,J'}}(k) = w_{J'}$.
Hence for each $j\in J'$, $(M_j,w_j) \models^\sharp \rho$. 

\fbox{$2. \impl 1.$}
For the preservation by $\F$-products, let $(M_i)_{i\in I}$ and $F\in
\F$ filter over $I$ and for any fixed $w\in \sem{M_F}$ assume
that there exists $J \in F$ and $k\in \sem{\mu_J}^{-1}(w)$ 
such that for each $j\in J$, $M_j \models^{k_j} \rho$.     
Let us take any $k_i \in \sem{M_i}$ for each $i\not\in J$ and let
$k_I$ be defined by $\sem{p_{I,i}}(k_I) = k_i$ for each $i\in I$. 
Since for each $j\in J$, $(M_j,k_j) \models^\sharp \rho$, by 2. it
follows that $(M_F,\sem{\mu_I}(k_I)) \models^\sharp \rho$.
Since $\sem{\mu_I}(k_I) = \sem{\mu_J}(k) = w$
it means that $M_F \models^w \rho$. 

For the preservation by $\F$-factors, let $(M_i)_{i\in I}$ and 
$F\in \F$ filter over $I$.
Let us assume that $M_F \models^w \rho$.
Let $k \in \sem{\mu_I}^{-1} (w)$. 
By (\ref{filt-prod-sharp-eqq}) of Prop.~\ref{filt-prod-sharp-prop} we
have that $(M_F,w)$  is the $F$-product of $((M_i,k_i))_{i\in I}$. 
By the hypothesis 2. there exists $J\in F$ such that for each $j\in
J$, $(M_j,k_j) \models^\sharp \rho$, which means $M_j \models^{k_j}
\rho$. 
\end{proof}

\begin{proposition}\label{pres-star-prop}
For any stratified institution $\I$ with $F$-products, if a sentence
$\rho$ is preserved by $F$-products in $\I$ then it is preserved by
$F$-products in $\I^*$ too. 
\end{proposition} 

\begin{proof}
Let us assume that $J'= \{ j\in I \mid M_j \models^* \rho \} \in F$
for $\{ \mu_J \co M_J \ra M_F \mid J \in F \}$ an $F$-product of a
family $( M_j )_{j\in I}$ of $\Sigma$-models. 
Let $w\in \sem{M_F}$. 
For any $k\in \sem{\mu_{J'}}^{-1}(w)$ and each $j\in J'$ we have that
$M_j \models^{k_j} \rho$ (since $M_j \models^* \rho$). 
Because $\rho$ is preserved by $F$-products in $\I$ it follows that
$M_F \models^w \rho$.
Hence $M_F \models^* \rho$. 
\end{proof}

According to \cite{upins,iimt} any institution in which all
its sentences are preserved by ultraproducts is m-compact.
Hence from Prop.~\ref{pres-star-prop} and \ref{pres-sharp-prop} we get
the following consequence.

\begin{corollary}\label{compact-cor}
Let $\I$ be a stratified institution with ultraproducts such that each
of its sentences are preserved by ultraproducts.
Then 
\begin{enumerate}

\item $\I^*$ is m-compact; and 

\item if in addition the ultraproducts are concrete then $\I^\sharp$
  is m-compact too.  
 
\end{enumerate}
\end{corollary}

The following consequence of Prop.~\ref{pres-sharp-prop} represents a
transfer of preservation results from ordinary institutions to
stratified institutions.  

\begin{corollary}\label{conn-cor}
In any  stratified institution $\I$ with concrete
$\F$-products 
\begin{enumerate}

\item both the sentences preserved by $\F$-products and those
  preserved by $\F$-factors are closed under conjunctions; 

\item if $\rho$ is preserved by $\F$-products then $\neg \rho$ is
  preserved by $\F$-factors; 

\item if $\rho$ is preserved by $\F$-factors and $\F$ contains only
  ultrafilters then $\neg \rho$ is preserved by $\F$-products; and

\item if $\F$ is closed under reductions, $\Mod(\chi)$ preserves
  $\F$-products, and $\rho$ is preserved by $\F$-products then 
    $(\exists \chi)\rho$ is preserved by $\F$-products. 

\end{enumerate} 
\end{corollary}

\begin{proof}
1., 2., 3. 
By Fact \ref{logic-sharp-fact}, the conjunction and negation coincide
in $\I$ and $\I^\sharp$.
By Prop.~\ref{pres-sharp-prop}, preservation by $\F$-products/factors
also coincides in $\I$ and $\I^\sharp$. 
The conclusions for 1., 2., 3. follow because by \cite{upins,iimt} the
considered preservation properties hold in general in any ordinary
institution and in particular in $\I^\sharp$.  

4. By Prop.~\ref{pres-sharp-prop} $\rho$ is preserved by $\F$-products
in $\I^\sharp$. 
By Cor.~\ref{chi-pres-cor} it follows that $\Mod^\sharp (\chi)$
preserves $\F$-products.
From \cite{upins,iimt} we know that in general, in any (ordinary)
institution, from such conditions it follows that $(\exists\chi) \rho$ is
preserved by $\F$-products.  
We apply this conclusion within $\I^\sharp$. 
By Fact \ref{logic-sharp-fact} (existential quantification coincide in
$\I$ and in $\I^\sharp$) and by Prop.~\ref{pres-sharp-prop} it now
follows that $(\exists\chi) \rho$ is preserved by $\F$-products in
$\I$.  
\end{proof}

The conclusions of Cor.~\ref{conn-cor} may be obtained directly
without reliance upon Prop.~\ref{pres-sharp-prop}. 
Some of them may be obtained under the slightly milder condition that
does not require the $F$-products to be concrete, however this
generality is largely meaningless in the applications because the
$F$-products are usually concrete (in fact we do not know examples of
$F$-products that are not concrete). 

\begin{proposition}\label{exists-pres-prop}
In any  stratified institution $\I$ with $\F$-products, if
$\F$ is closed under reductions, $\Mod(\chi)$ invents $\F$-products,
and $\rho$ is preserved by $\F$-factors then $(\exists \chi)\rho$
is preserved by $\F$-factors.  
\end{proposition}

\begin{proof}
Let $\chi \co \Sigma \ra\Sigma'$ signature morphism, let $F\in \F$,
and let $\{ \mu_J \co M_J \ra M_F \mid J \in F \}$ be an $F$-product
of a family $(M_i)_{i\in I}$ of $\Sigma$-models.
Assume that $M_F \models^w (\exists \chi)\rho$. 

It follows that there exists $M'$ and $w'$ such that 
$M_F = \Mod(\chi)(M')$, 
$w' \in \sem{M'}_\chi^{-1}(w)$, and $M' \models^{w'} \rho$. 
By the inventing condition there exists $J\in F$ and an $F|_J$-product 
$\{ \mu'_{J'} \co M'_{J'} \ra M' \mid J' \in F|_J \}$ of a family
$(M'_j)_{j\in J}$ of $\Sigma'$-models such that $\Mod(\chi)(M'_j) =
M_j$ for each $j\in J$ and for all $J'' \subseteq
J' \in F|_J$ we have that $\Mod(\chi)(p'_{J' \supseteq J''}) = p_{J'
  \supseteq J''}$ and $\Mod(\chi)(\mu'_{J'}) = \mu_{J'}$.
Since $\rho$ is preserved by $\F$-factors there exists $J'\in F|_J$
and $k'\in \sem{\mu'_{J'}}^{-1}_{\Sigma'} (w')$ such that $M'_j
\models^{k'_j} \rho$ for each $j\in J'$. 
Let $k =\sem{M'_{J'}}_\chi (k')$. 
For each $j\in J'$ we have the following:
$$\begin{array}{rll}
k_j = & \sem{\Mod(\chi)(p_{J',j})}(k)
      & (\text{by the definition of }k_j) \\
= & \sem{\Mod(\chi)(p_{J',j})}(\sem{M'_{J'}}_\chi (k'))
  & (\text{by the definition of }k) \\
= & \sem{M'_j}_\chi(\sem{p_{J',j}}(k')) 
  & (\text{by (\ref{diag1})}) \\
= & \sem{M'_j}_\chi (k'_j) 
  & (\text{by the definition of }k'_j).
\end{array}$$
Since $\Mod(\chi)(M'_j) = M_j$ we get that 
$M_j \models^{k_j} (\exists \chi)\rho$.
It remains to show that $\sem{\mu_{J'}}_\Sigma (k) = w$, which holds
by the following calculation:
$$\begin{array}{rll}
\sem{\mu_{J'}}_\Sigma (k) = 
  & \sem{\mu_{J'}}_\Sigma (\sem{M'_{J'}}_\chi (k')) 
  & (\text{by the definition of )k})\\
= & \sem{M'}_\chi \big( \sem{\mu'_{J'}}_{\Sigma'} (k') \big) 
  & (\text{by (\ref{diag1})})\\
= & \sem{M'}_\chi (w') 
  & (\text{since }k'\in \sem{\mu'_{J'}}^{-1}_{\Sigma'} (w'))\\
= & w. &
\end{array}$$
\end{proof}

\begin{proposition}\label{modal-prop}
Let $\I$ be a stratified institution endowed with a frame extraction
$L\co \Sign^{\I} \ra \Sign^{\REL}$, $\Fr \co \Mod^\I \Ra
L;\Mod^{\REL}$. 
Assume that $\I$ has $F$-products for a filter $F$ over a set $I$.  
\begin{enumerate}

\item If $\Fr_\Sigma$ preserves direct products and
  $\rho_1,\dots,\rho_n$ are preserved by $F$-products then any  
$\lambda$-possibility $\langle\lambda\rangle(\rho_1,\dots\rho_n)$ is
also preserved by $F$-products.  

\item If $\Fr_\Sigma$ preserves $F$-products and
  $\rho_1,\dots,\rho_n$ are preserved by $F$-factors then any
  $\lambda$-possibility $\langle\lambda \rangle (\rho_1,\dots\rho_n)$
  is also preserved by $F$-factors.  

\end{enumerate}
\end{proposition}

\begin{proof}
1. We consider an $F$-product 
$\{ \mu_J \co M_J \ra M_F \mid J \in F \}$ for a family $(M_i)_{i\in
  I}$ of $\Sigma$-models and assume that there exists $J\in F$ and
$k\in \sem{\mu_J}^{-1}(w)$ such that for each $j\in J$, 
$M_j \models^{k_j} \langle\lambda\rangle(\rho_1,\dots\rho_n)$. 
We have to prove that 
$M_F \models^w \langle\lambda\rangle(\rho_1,\dots,\rho_n)$, i.e. that
there exists $(w,w_1,\dots,w_n) \in (\Fr_\Sigma (M))_\lambda$ such
that $M_F \models^{w_i} \rho_i$ for each $1\leq i\leq n$. 

For each $j\in J$, 
$M_j \models^{k_j} \langle\lambda\rangle(\rho_1,\dots\rho_n)$
means that there exists 
$(k_j,k_j^1,\dots,k_j^n)\in (\Fr_\Sigma (M_j))_\lambda$ such that 
$M_j \models^{k_j^i} \rho_i$ for each $1\leq i\leq n$. 
Since $\Fr_\Sigma$ preserves products we have that 
$\{ \Fr_\Sigma (p_{J,j}) \co \Fr_\Sigma(M_J) \ra \Fr_\Sigma(M_j) \mid
j\in J \}$ is direct product in $\Mod^\REL (L(\Sigma))$. 
Hence for each $1\leq i\leq n$, there exists $k^i\in \sem{M_J}$ such
that $\sem{p_{J,j}}(k^i) = k^i_j$ for each $j\in J$. 
We define $w_i = \sem{\mu_J}(k^i)$. 

By the direct product property of $\Fr_\Sigma(M_J)$ in
$\Mod^\REL (L(\Sigma))$ we have that 
$(k_j,k_j^1,\dots,k_j^n)\in (\Fr_\Sigma(M_j))_\lambda$ for each
$j\in J$ implies that $(k,k^1,\dots,k^n)\in
(\Fr_\Sigma(M_J))_\lambda$.  
Since $\Fr_\Sigma(\mu_J)$ is a homomorphism of
$\Mod^\REL (L(\Sigma))$-models it follows that 
$(w,w_1,\dots,w_n)= 
(\sem{\mu_J}(k),\sem{\mu_J}(k^1),\dots,\sem{\mu_J}(k^n))
\in (\Fr_\Sigma(M_F))_\lambda$. 

That for each $1 \leq i\leq n$, $M_F \models^{w_i} \rho_i$, follows
from the hypothesis that $\rho_i$ is preserved by $F$-products and
because $k^i \in \sem{\mu_J}^{-1}(w_i)$ and $M_j \models^{k^i_j}
\rho_i$ for each $j\in J$.  

2. We consider an $F$-product 
$\{ \mu_J \co M_J \ra M_F \mid J \in F \}$ for a family $(M_i)_{i\in
  I}$ of $\Sigma$-models and assume that 
$M_F \models^w \langle\lambda\rangle (\rho_1,\dots,\rho_n)$.
We have to prove that there exists $J\in F$ and
$k\in \sem{\mu_J}^{-1}(w)$ such that for each $j\in J$, 
$M_j \models^{k_j} \langle\lambda\rangle (\rho_1,\dots,\rho_n)$,
i.e. that there exists $(k_j,k_j^1,\dots,k_j^n)\in
(\Fr_\Sigma(M_j))_\lambda$ such that for each $1\leq i\leq n$, 
$M_j \models^{k_j^i} \rho_i$. 

From $M_F \models^w \langle\lambda\rangle (\rho_1,\dots,\rho_n)$
it follows that there exists $(w,w_1,\dots,w_n)\in
(\Fr_\Sigma(M_F))_\lambda$ such that $M_F \models^{w_i} \rho_i$ for each
$1\leq i\leq n$. 
By the hypothesis that each $\rho_i$ is preserved by $F$-factors, this 
means there exists $J_i\in F$ and  $l^i \in \sem{\mu_{J_i}}^{-1}
(w_i)$ such that $M_j \models^{l_j^i} \rho_i$ for each $j\in J_i$.  

Since $\Fr_\Sigma$ preserves $F$-products it follows that 
$\{ \Fr_\Sigma(\mu_J) \co \Fr_\Sigma(M_J) \ra \Fr_\Sigma(M_F) \mid
J \in F \}$ is an $F$-product of $(\Fr_\Sigma(M_j))_{j\in I}$ in
$\Mod^\REL (L(\Sigma))$. 
Hence, $(w,w_1,\dots,w_n)\in (\Fr_\Sigma(M_F))_\lambda$ implies that
there exists $J'\in F$ and 
$(v,v_1,\dots,v_n)\in (\Fr_\Sigma(M_{J'}))_\lambda$ with
$\sem{\mu_{J'}}(v) = w$ and $\sem{\mu_{J'}}(v_i) = w_i$ for each
$1\leq i \leq n$.

Let us take $J'' = J' \cap \bigcap_{1\leq i \leq n} J_i$. 
Since filters are closed under intersections, it follows that 
$J'' \in F$. 
For each $1\leq i\leq n$ we have that 
\[
\sem{\mu_{J''}}(\sem{p_{J_i \supseteq J''}}(l^i)) = 
\sem{\mu_{J_i}}(l^i) = w_i = \sem{\mu_{J'}}(v_i) =
\sem{\mu_{J''}}(p_{J'\supseteq J''}(v_i)).
\]
Since $\{ \Fr_\Sigma(\mu_J) \co \Fr_\Sigma(M_J) \ra \Fr_\Sigma(M_F) \mid
J \in F \}$ is an $F$-product, which means it is a particular directed
co-limit, it follows that there exists $J \subseteq J''$ such that  
$\sem{p_{J_i \supseteq J}}(l^i) = \sem{p_{J'\supseteq J}}(v_i)$ for
each $1\leq i \leq n$. 

For each $1\leq i\leq n$ we define 
$k^i = \sem{p_{J_i\supseteq J}}(l^i)= \sem{p_{J'\supseteq J}}(v_i)$. 
We also let $k = \sem{p_{J' \supseteq J}}(v)$. 
\begin{itemize}\smallitems

\item[--] Since $(v,v_1,\dots,v_n)\in (\Fr_\Sigma(M_{J'}))_\lambda$,
  by the homomorphism property of $\Fr_\Sigma (p_{J'\supseteq J})$ it follows
  that $(k,k^1,\dots,k^n) \in (\Fr_\Sigma(M_J))_\lambda$. 
  By the homomorphism property of each $p_{J,j}$ it further follows
  that $(k_j,k_j^1,\dots,k_j^n) \in (\Fr_\Sigma(M_j))_\lambda$ for
  each $j\in J$. 

\item[--] Note that for each $1\leq i\leq n$ and each $j\in J$ 
\[
l^i_j = \sem{p_{J_i,j}}(l^i) = \sem{p_{J,j}}(\sem{p_{J_i \supseteq J}}(l^i))
= \sem{p_{J,j}}(k^i) = k^i_j.
\]
Since we know that $M_j \models^{l^i_j} \rho_i$ it means that $M_j
\models^{k^i_j} \rho_i$ for each $j\in J$. 

\end{itemize}
\end{proof}

\begin{proposition}\label{hybrid-prop}
Let $\I$ be a stratified institution endowed with a nominals
extraction $N \co \Sign^\I \ra \Sign^\SETC$, $\Nm \co \Mod^\I \Ra
N;\Mod^\SETC$. 
Assume that $\I$ has $F$-products for a filter $F$ over a set $I$. 
For any signature $\Sigma$ and any $i\in N(\Sigma)$, 
\begin{enumerate}

\item If $\Nm_\Sigma$ preserves direct products then $\sen{i}$ is
  preserved by $F$-products. 

\item If $\Nm_\Sigma$ preserves $F$-products then $\sen{i}$ is
  preserved by $F$-factors. 

\item If $\rho$ is preserved by $F$-products then each sentence 
$@_i \rho$ is preserved by $F$-products too. 

\item If $\Nm_\Sigma$ preserves $F$-products and $\rho$ is
  preserved by $F$-factors then each sentence $@_i \rho$ is preserved
  by $F$-products too.     

\end{enumerate}
\end{proposition}

\begin{proof}
We consider $\{ \mu_J \co M_J \ra M_F \mid J \in F \}$
an $F$-product a family $(M_j)_{j\in I}$ in $\Mod(\Sigma)$. 

1. Let us assume that there exists $J\in F$ and $k\in
\sem{\mu_J}^{-1}_\Sigma (w)$ such that $M_j \models^{k_j} \sen{i}$ for
each $j\in J$.  
This means for each $j\in J$
\begin{equation}\label{nom-eq1}
(\Nm_\Sigma (M_j))_i = k_j = \Nm_\Sigma (p_{J,j}) (k).
\end{equation} 
Also, by the homomorphism property of $\Nm_\Sigma (p_{J,j})$ we have
that for  each $j\in J$
\begin{equation}\label{nom-eq2}
(\Nm_\Sigma (M_j))_i = \Nm_\Sigma (p_{J,j}) ((\Nm_\Sigma (M_J))_i).
\end{equation}
Since $\Nm_\Sigma$ preserves direct products, 
from (\ref{nom-eq1}) and (\ref{nom-eq2}) it follows that 
$(\Nm_\Sigma (M_J))_i = k$. 
We have that 
$$\begin{array}{rll}
(\Nm_\Sigma (M_F))_i = & 
  \Nm_\Sigma (\mu_J) ((\Nm_\Sigma (M_J))_i) &
  (\text{by the homomorphism property of }\Nm_\Sigma (\mu_J)) \\
 = & \Nm_\Sigma (\mu_J) (k) & = w,
\end{array}$$
which means $M_F \models^w \sen{i}$. 

2. Let us assume that $M_F \models^w \sen{i}$, which means 
$(\Nm_\Sigma (M_F))_i =w$. 
Since $\Nm_\Sigma$ preserves $F$-products, 
$\{ \Nm_\Sigma (\mu_J) \co \Nm_\Sigma (M_J) \ra \Nm_\Sigma (M_F) \mid
J \in F \}$ is a directed co-limit, hence there exists $J\in F$ such
that $\Nm_\Sigma (\mu_J) ((\Nm_\Sigma (M_J))_i) = (\Nm_\Sigma (M_F))_i$. 
Let $k = (\Nm_\Sigma (M_J))_i$. 
For each $j\in J$, by the homomorphism property of $\Nm_\Sigma
(p_{J,j})$ it follows that $k_j = \Nm_\Sigma (p_{J,j}) (k) =
\Nm_\Sigma (p_{J,j}) ((\Nm_\Sigma (M_J))_i) = (\Nm_\Sigma
(M_j))_i$ which means $M_j \models^{k_j} \sen{i}$.   

3. Let us assume that there exists $J\in F$ and $k\in
\sem{\mu_J}^{-1}_\Sigma (w)$ such that $M_j \models^{k_j} @_i \rho$
for each $j\in J$, which just means $M_j
\models^{(\Nm_\Sigma (M_j))_i} \rho$ for each $j\in J$.
Since by the homomorphism property of $\Nm_\Sigma (\mu_J)$ and of 
$\Nm_\Sigma (p_{J,j})$, for each $j\in J$, we have that
$\Nm_\Sigma (\mu_J) ((\Nm_\Sigma (M_J))_i) = (\Nm_\Sigma (M_F))_i$ and 
that $\Nm_\Sigma (p_{J,j}) ((\Nm_\Sigma (M_J))_i) =
(\Nm_\Sigma (M_j))_i$, respectively, and because by hypothesis $\rho$
is preserved by $F$-products it follows that 
$M_F \models^{(\Nm_\Sigma (M_F))_i} \rho$ which means $M_F \models^w
@_i \rho$. 

4. Let us assume $M_F \models^w @_i \rho$, which means 
$M_F \models^{(\Nm_\Sigma (M_F))_i} \rho$. 
It is enough to show that there exists $J\in F$ such that $M_j
\models^{(\Nm_\Sigma (M_j))_i} \rho$ for each $j\in J$. 
\begin{itemize}\smallitems

\item[--]
Since $\Nm_\Sigma$ preserves $F$-products, 
$\{ \Nm_\Sigma (\mu_J) \co \Nm_\Sigma (M_J) \ra \Nm_\Sigma (M_F) \mid
J \in F \}$ is a directed co-limit, hence there exists $J'\in F$ such
that $\Nm_\Sigma (\mu_{J'}) ((\Nm_\Sigma (M_{J'}))_i) = (\Nm_\Sigma
(M_F))_i$.  

\item[--]
By the hypothesis that $\rho$ is preserved by $F$-factors, it follows
that there exists $J''\in F$ and 
$k'' \in \sem{\mu_{J''}}^{-1}_\Sigma ((\Nm_\Sigma (M_F))_i)$ such that
$M_j \models^{k''_j} \rho$ for each $j\in J''$. 

\end{itemize}
Since $\sem{\mu_{J'}}_\Sigma ((\Nm_\Sigma (M_{J'}))_i) = 
\sem{\mu_{J''}}_\Sigma (k'')$ and because 
$\{ \sem{\mu_J}_\Sigma \co \sem{M_J}_\Sigma \ra \sem{M_F}_\Sigma \mid
J \in F \}$ is a directed co-limit, there exists $J \subseteq J'\cap
J'' \in F$ such that 
\begin{equation}\label{at-eq1}
\sem{p_{J'\supseteq J}}_\Sigma ((\Nm_\Sigma (M_{J'}))_i) = 
\sem{p_{J''\supseteq J}}_\Sigma (k'').
\end{equation}
For each $j\in J$ we have that 

\vspace{.5em}

\noindent
$(\Nm_\Sigma (M_j))_i =$
\vspace{-.5em}
$$\begin{array}{rrll}
& = &   \Nm_\Sigma (p_{J,j}) ((\Nm_\Sigma (M_{J}))_i) & 
  (\text{by the homomorphism property of }\Nm_\Sigma (p_{J,j})) \\
& = & \Nm_\Sigma (p_{J,j}) \big(\Nm_\Sigma (p_{J'\supseteq J})
       ((\Nm_\Sigma (M_{J'}))_i)\big) &
  (\text{by the homomorphism property of }\Nm_\Sigma (p_{J'\supseteq J}))\\
& = & \Nm_\Sigma (p_{J,j}) \big(\Nm_\Sigma (p_{J''\supseteq J}) 
       (k'')\big) & 
  (\text{by (\ref{at-eq1})}) \\
&  =  & \sem{p_{J'',j}}_\Sigma (k'') & = k''_j.
\end{array}$$
Hence for each $j\in J$, $M_j \models^{(\Nm_\Sigma (M_j))_i} \rho$.
\end{proof}

Note that from the six preservation results included in Prop.~\ref{modal-prop}
and \ref{hybrid-prop}, one does not assume anything on the
frame/nominals extraction, two assume that the respective extractions
preserve direct products, and three that the they preserve
$F$-products.  

The preservation results of Cor.~\ref{conn-cor} and of
Prop.~\ref{exists-pres-prop}--\ref{hybrid-prop} may be applied for
lifting preservation properties from simpler to more complex
sentences.
They can be used at the induction step when establishing preservation
properties by induction on the structure of the sentences. 
The following result and its corollary constitute a general approach
to the base case of such induction proofs, that in general corresponds
to the atomic sentences.  

\begin{lemma}\label{atom-lem}
Let $(\Phi,\alpha,\beta) \co \B' \ra \B$ be an institution morphism
such that each $\beta_\Sigma$ preserves $F$-products.
Then for any $\Phi(\Sigma)$-sentence $\rho$ that is preserved by
$F$-products/factors, the $\Sigma$-sentence $\alpha_\Sigma (\rho)$ is
preserved by $F$-products/factors. 
\end{lemma}

\begin{proof}
Let us assume an $F$-product 
$\{ \mu'_J \co M'_J \ra M'_F \mid J\in F \}$ of a family 
$(M'_i)_{i\in I}$ of $\Sigma$-models for a $\B'$-signature $\Sigma$. 
By hypothesis we have that 
$\{ \beta_\Sigma (\mu'_J) \co \beta_\Sigma (M'_J) \ra \beta_\Sigma
(M'_F) \mid J \in F \}$ is an $F$-product of 
$(\beta_\Sigma (M'_i))_{i\in I}$ in $\Mod^{\B}(\Phi(\Sigma))$. 

For the preservation by $F$-products, let us assume $J\in F$ such that
$M'_i \models_\Sigma \alpha_\Sigma (\rho)$ for each $i\in J$. 
By the satisfaction condition of $(\Phi,\alpha,\beta)$ this means 
$\beta_\Sigma (M'_i) \models_{\Phi(\Sigma)} \rho$ for each $i\in J$,
hence because $\rho$ is preserved by $F$-products, 
$\beta_\Sigma (M'_F) \models_{\Phi(\Sigma)} \rho$.  
By the satisfaction condition of $(\Phi,\alpha,\beta)$ it follows that
$M'_F \models_\Sigma \alpha_\Sigma (\rho)$. 

For the preservation by $F$-factors, let us assume that $M'_F
\models_\Sigma \alpha_\Sigma (\rho)$. 
By the satisfaction condition of $(\Phi,\alpha,\beta)$ it follows that
$\beta_\Sigma (M'_F) \models_{\Phi(\Sigma)} \rho$. 
Since $\rho$ is preserved by $F$-factors, there exists $J\in F$ such
that $\beta_\Sigma (M'_i) \models_{\Phi(\Sigma)} \rho$ for each $i\in
J$. 
By the satisfaction condition of $(\Phi,\alpha,\beta)$ we obtain that
$M'_i \models_\Sigma \alpha_\Sigma (\rho)$ for each $i\in J$. 
\end{proof}

The following is an immediate consequence of
Prop.~\ref{pres-sharp-prop} and Lemma \ref{atom-lem}, which is
applicable in concrete situations. 

\begin{corollary}\label{atomic-cor}
Let $\I$ be a  stratified institution with concrete $\F$-products.  
Let $(\Phi,\alpha,\beta) \co \I^\sharp \ra \B$ be an institution
morphism such that each $\beta_\Sigma$ preserves $\F$-products.  
Then for each $\Phi(\Sigma)$-sentence $\rho$ that is preserved by
$\F$-products/factors, $\alpha_\Sigma (\rho)$ is preserved by
$\F$-products/factors in $\I$.  
\end{corollary}

Now we can put together the results of this section and apply them to
our concrete benchmark examples. 

\begin{corollary}
Let $\I \in \{
\MPL,\MFOL,\HPL,\HFOL,\MMPL,\MHPL,\MMFOL,\MHFOL,\HHPL,\OFOL,$
$\MOFOL,\HOFOL,\HMOFOL \}$. 
Then in $\I$ each sentence is preserved by all ultraproducts and
ultrafactors.  
Consequently $\I^\sharp$ and $\I^*$ are m-compact and in addition
$\I^\sharp$ is compact. 
\end{corollary}

\begin{proof}
The first conclusion is proved by induction on the structure of
$\I$-sentences through application of the preservation results of
Cor.~\ref{atomic-cor}, \ref{conn-cor}, Prop.~\ref{exists-pres-prop},
\ref{modal-prop}, and \ref{hybrid-prop} as follows. 

From Ex.~\ref{f-prod-ex} let us note that $\I$ has concrete
$F$-products for any filter $F$.   

The base case of our induction proof on the structure of the
$\I$-sentences is represented, with the exception of $\HHPL$, only by
atomic sentences.   
These atomic sentences may be of two kinds, either atomic
sentences of $\PL$ or $\FOL$, or else $\sen{i}$. 
In the case of $\HHPL$, besides $\sen{i^1}$ at the base case we also
have the sentences of the $\HPL$ corresponding to the lower layer of
hybridization. 
For the case when the sentence is a not a nominal sentence, 
we apply Cor.~\ref{atomic-cor}.
Let $\APL$ and $\AFOL$ denote the sub-institutions of $\PL$
(propositional logic) and of $\FOL$ (first order logic),
respectively, that have only the atoms as their sentences.
Let $\B$ be $\HPL$ when $\I = \HHPL$, $\APL$ when $\I\in \{
\MPL,\HPL,\MMPL,\MHPL \}$ and $\AFOL$ otherwise.  
The institution morphism $(\Phi,\alpha,\beta) \co \I \ra \B$ is
defined as follows:
\begin{itemize}\smallitems

\item $\Phi$ forgets the modalities symbols $\Lambda$ when 
$\I \in \{ \MMPL,\MHPL,\MMFOL,\MHFOL \}$ and the
nominals symbols when $\I \in \{ \HPL,\HFOL,\MHPL,\MHFOL,\HHPL
\}$\footnote{In the $\HHPL$ case we have $\Phi(\Nom^0,\Nom^1,P) =
  (\Nom^0,P)$.} and is identity otherwise;  

\item $\alpha$ is just the inclusion of the sentences of $\APL$ or of
  $\AFOL$ as atomic sentences of $\I$; and

\item $\beta_\Sigma (M,w) = M^w$.

\end{itemize}
The Satisfaction Condition for $(\Phi,\alpha,\beta)$ is an immediate
consequence of the satisfaction of atomic sentences in $\I$ (or of the
satisfaction of the $\HPL$-sentences in $\HHPL$) and of the definition
of $\models^\sharp$ (see Fact \ref{sharp-institution-fact}).   

Now we establish that each $\beta_\Sigma$ preserves all $F$-products. 
By Prop.~\ref{filt-prod-sharp-prop} we know that $F$-products in
$\I^\sharp$ are of the form 
\[
\{ (\mu_J,w_J) \co (M_J,w_J) \ra (M_F, \sem{\mu_I}(w_I)) \mid J \in F \}.
\]
According to the definition of $\beta$, we have to show that 
\begin{equation}\label{beta-equation}
\{ \mu^{w_J}_J \co M_J^{w_J} \ra M_F^{\sem{\mu_I}(w_I)} \mid J \in F \}
\end{equation}
is an $F$-product too. 
Without any loss of generality we may further assume that $M_J$ are
cartezian products. 
Note that $w_J = (w_j)_{j\in J}$ when the $\I$-models are Kripke
models and $w_J ; p_{J,j} = w_j$ in the other cases. 
It follows that $M_J^{w_J}$ is the product of 
$\{ M_j^{w_j} \mid j\in J \}$. 
When the $\I$-models are Kripke models, from the construction of
$F$-products of Kripke models, by Lemma 11.11 of \cite{iimt} (the same
with Lemma 1 of \cite{ks}) it follows that (\ref{beta-equation}) is an
$F$-product of $(M^{w_j}_j)_{j\in I}$.  
When $\I\in \{ \OFOL,\MOFOL,\HOFOL,\HMOFOL \}$ then the argument that
(\ref{beta-equation}) is an $F$-product is much simpler because 
$\{ \mu_J \co M_J \ra M_F \mid J \in F \}$ is an $F$-product of
$\FOL$ $(F,P)$-models and (\ref{beta-equation}) is just an expansion
of this to $(F+X,P)$.\footnote{Note that in this argument $F$ is
  overloaded, it means both the filter and the family of function
  symbols of the signature.}

When $\B \not= \HPL$ then all $\B$-sentences are atoms, hence according
to \cite{upins,iimt} they are `finitary basic sentences' and
consequently are preserved by all $F$-products and all $F$-factors.  
When $\B = \HPL$ then we have to use the conclusion of this
corollary for $\I = \HPL$, that all $\HPL$-sentences are preserved
by ultraproducts. 
This completes the set of conditions for applying
Cor.~\ref{atomic-cor}, which gets us to the conclusion that, apart of
the nominal sentences $\sen{i}$, all sentences at the base case 
are preserved by ultraproducts and ultrafactors. 
For the sentences $\sen{i}$ we apply the relevant part of
Prop.~\ref{hybrid-prop}. 
For this we have just to note that the condition that
$\Nm_\Sigma$ preserves direct products and ultraproducts is
covered by the fact that $\I$ has concrete $F$-products.  
This covers the base case of our induction proof.   

According to the definition of satisfaction in $\I$ all $\I$-sentences
are built by iterative application of external Boolean connectives, quantifiers,
modalities, $@_i$, from atoms when $\I \not= \HHPL$ and from
$\HPL$-sentences plus  $\sen{i^1}$ when $\I = \HHPL$.  
Hence for the induction step part of the proof, we have to check the
conditions of Cor.~\ref{conn-cor}, Prop.~\ref{exists-pres-prop},
\ref{modal-prop}, and \ref{hybrid-prop}.   
The preservation of direct products and of ultraproducts by
$\Fr_\Sigma, \Nm_\Sigma$ is a direct consequence of the construction
of filtered products of Kripke models.
Because the class of all ultrafilters is closed
under reductions, it remains only to show that, when applicable, 
for each signature extension $\chi$ with first order variables or with
nominals variables, $\Mod(\chi)$ preserves and invents ultraproducts. 

The preservation property holds for all $F$-products as follows. 
First we have to notice it for the direct products. 
When $\I\in \{ \OFOL,\MOFOL,\HOFOL,\HMOFOL \}$ this is just a matter
of preservation of direct products of $\FOL$ models by reducts
forgetting interpretations of constants, which is obvious.  
When the $\I$-models are Kripke models, this is a consequence of the
fact that whenever we expand a direct product $(W,M)$ of a family
$(W_i,M_i)_{i\in I}$ of reducts of Kripke models 
$(W'_i,M'_i)_{i\in I}$ with an interpretation of a new constant $x$ in
$(W,M)$ by $W'_x = ((W'_i)_x)_{i\in I}$ when $x$ is 
nominal or by $M'_x = ((M'_i)_x)_{i\in I}$ when $x$ is a first order
constant, this yields a direct product of $(W'_i,M'_i)_{i\in
  I}$.\footnote{Note that here, in order to simplify the discussion,
  we implicitly assumed cartezian products, which is no loss of
  generality, and that since in all situations for $\I$ the
  interpretation of first order constants are shared in all possible
  worlds we may have a notation such as $M_x$ instead of $M^w_x$.}    
The argument is completed by noting that the directed co-limit 
component of any $F$-product is preserved by reducts corresponding to
signature expansions $\chi$ with nominal or first order variables as a
consequence of the fact that any model homomorphism 
$\Mod(\chi)(M') \ra N$ may be expanded uniquely to a model
homomorphism $M' \ra N'$.\footnote{At the level of abstract
  institutions, in \cite{iimt} this property is called
  `quasi-representability'; moreover \cite{iimt} gives a general
  result that quasi-representable signature morphisms always
  preserve directed co-limits.}    
This property holds both in the simpler case when the $\I$-models are
$\FOL$-models but also in the case when they are Kripke models; in
the latter situation, in the case of the first order variables the
uniqueness of $N'$ relies upon the fact that interpretations of the
underlying carriers and of the first order constants are shared across
the possible worlds.   

Now we show that the inventing property holds in the complete form for
all $F$-products. 
Let $\chi \co \Sigma \ra \Sigma'$ be a signature extension with
nominal or first order variables and let $\{ \mu_J \co M_J \ra M_F
\mid J \in F \}$ be an $F$-product of a family $(M_i)_{i\in I}$ of
$\Sigma$-models. 
Let $N'$ be any $\chi$-expansion of $M_F$. 
Since $\mu_I \co M_I \ra M_F$ is surjective\footnote{In the case of
  Kripke models this means that all its components are surjective.}
there exists $M'_I$ a $\chi$-expansion of $M_I$ such that $\mu_I$ is a
$\Sigma'$-model homomorphism $M'_I \ra N'$.
For each $i\in I$ we let $M'_i$ be the $\chi$-expansion of $M_i$ such
that $p_{I,i} \co M'_I \ra M'_i$ is $\Sigma'$-homomorphism. 
This yields a lifting of $\{ \mu_J \co M_J \ra M_F \mid J \in F \}$ to
a co-cone $\{ \mu_J \co M'_J \ra N' \mid J \in F \}$ over a directed
diagram of projections in $\Mod(\Sigma')$. 
For any other co-cone $\{ \nu_J \co M'_J \ra N'' \mid J \in F \}$
we let $h \co M_F \ra \Mod(\chi)(N'')$ be the unique mediating
homomorphism given by the co-limit property of 
$\{ \mu_J \co M_J \ra M_F \mid J \in F \}$.
It remains to show that $h \co N' \ra N''$ is a homomorphism of
$\Sigma'$-models. 
This follows by virtue of the fact that $\mu_I ; h = \nu_I$ and
because $\nu_I$ is a homomorphism of $\Sigma'$-models. 

The m-compactness properties of $\I^*$ and $\I^\sharp$ follow
immediately from the first part of this corollary via
Cor.~\ref{compact-cor}.  
The compactness property of $\I^\sharp$ follows from the general
result that compactness and m-compactness are equivalent properties in
institutions that have external negations and conjunctions (see
\cite{iimt}), which by Fact \ref{logic-sharp-fact} is the case for all
institutions $\I^\sharp$ considered here.  
\end{proof}

\section{Conclusions}

In this paper we have showed that the stratified
institutions of \cite{strat} may serve as a general fully abstract model
theoretic framework for modal logical systems. 
We have shown that stratified institutions allow for an
abstract semantics for modalities, nominals, and satisfaction operator
($@$); in each of these cases we had been able to employ the minimal
structures supporting the corresponding semantics.
Within this context we have developed a general ultraproducts method,
including a general {\L}o\'{s} theorem, applicable to a wide variety
of modal logical systems. 
Compactness results have have been derived from this ultraproducts
method. 
The concepts introduced and the results developed have been applied to
a series of concrete benchmark examples that include both well known
and quite unconventional modal logical systems from logic and
computing. 
Due to the very high level of generality of our developments, without
commitment to explicit forms of Kripke semantics, our work may be
easily applicable to a multitude of new unconventional logical
systems. 
Moreover it may constitute a starting point for a deep institution
theoretic approach to a dedicated model theory for modal logical
systems in the style of \cite{iimt}.  
  
\paragraph{Acknowledgements}
This work has been supported by a grant of the Romanian 
National Authority for Scientific Research, CNCS-UEFISCDI, project number 
PN-II-ID-PCE-2011-3-0439.

\bibliographystyle{plain}

\end{document}